\documentclass[12pt]{article}
\usepackage{latexsym, amsmath, amsthm, amsfonts}
\usepackage{bbm}
\usepackage{enumerate, amssymb, xspace}
\usepackage[mathscr]{eucal}
\usepackage{graphicx}
\usepackage{rotating}
\usepackage{hyperref}
\usepackage{color,colordvi}

\usepackage{mathtools}   

\usepackage{tikz}
\usepackage{tikz-cd}
\usetikzlibrary{matrix,arrows,tqft,calc,decorations,arrows,shapes}
\usetikzlibrary{decorations.markings}
\usetikzlibrary{arrows.meta}  

\usepackage{makeidx}
\usepackage{pifont}
\usepackage{psfrag}
\usepackage{stackrel}

\newcommand\scopeArrow[2] {\begin{scope}[decoration={markings,mark=at position #1
                 with \arrow{#2}}]}  

             \newcommand\Cite[2] {\cite[#1]{#2}}

\catcode`\@=11
\newif\if@fewtab\@fewtabtrue
{\count255=\time\divide\count255 by 60
\xdef\hourmin{\number\count255}
\multiply\count255 by-60\advance\count255 by\time
\xdef\hourmin{\hourmin:\ifnum\count255<10 0\fi\the\count255}}
\def\ps@draft{\let\@mkboth\@gobbletwo
    \def\@oddfoot{\hbox to 7 cm{\tiny \versionno
       \hfil}\hskip -7cm\hfil\rm\thepage \hfil {\tiny\draftdate}}
    \def\@oddhead{}
    \def\@evenhead{}\let\@evenfoot\@oddfoot}
\def\draftdate{\number\month/\number\day/\number\year\ \ \ \hourmin }

\def\citen#1{\if@filesw \immediate\write \@auxout {\string\citation{#1}}\fi%
\@tempcntb\m@ne \let\@h@ld\relax \def\@citea{}%
\@for \@citeb:=#1\do {\@ifundefined {b@\@citeb}%
    {\@h@ld\@citea\@tempcntb\m@ne{\bf ?}%
    \@warning {Citation `\@citeb ' on page \thepage \space undefined}}%
    {\@tempcnta\@tempcntb \advance\@tempcnta\@ne
    \setbox\z@\hbox\bgroup\ifcat0\csname b@\@citeb \endcsname \relax
    \egroup \@tempcntb\number\csname b@\@citeb \endcsname \relax
    \else \egroup \@tempcntb\m@ne \fi \ifnum\@tempcnta=\@tempcntb
    \ifx\@h@ld\relax \edef \@h@ld{\@citea\csname b@\@citeb\endcsname}%
    \else \edef\@h@ld{\hbox{--}\penalty\@highpenalty
    \csname b@\@citeb\endcsname}\fi
    \else \@h@ld\@citea\csname b@\@citeb \endcsname \let\@h@ld\relax \fi}%
\def\@citea{,\penalty\@highpenalty\hskip.13em plus.13em minus.13em}}\@h@ld}
\def\@citex[#1]#2{\@cite{\citen{#2}}{#1}}%
\def\@cite#1#2{\leavevmode\unskip\ifnum\lastpenalty=\z@\penalty\@highpenalty\fi%
  \ [{\multiply\@highpenalty 3 #1%
  \if@tempswa,\penalty\@highpenalty\ #2\fi}]}   %
\makeatother 
\catcode`\@=12

\def\act           {\,{.}\,}
\def\Act           {{.}}
\def\be            {\begin{equation}}
\def\bearl         {\begin{array}{l}}
\def\bearll        {\begin{array}{ll}}
\def\bfF           {{g}}
\def\bfH           {{h}}

\def\boti          {\,{\boxtimes}\,}
\def\C             {{\ensuremath\calc}}
\def\cala          {{\mathcal A}}
\def\Cala          {{\!\mathcal A}}
\def\calb          {{\mathcal B}}
\def\calc          {{\mathcal C}}

\def\calcm         {\calc^{\star_{\phantom l}}_\Calm}

\def\calm          {{\mathcal M}}
\def\Calm          {{\!\mathcal M}}
\def\calmopp       {{\calm\opp_{\phantom|}}}

\def\caln          {{\mathcal N}}

\def\calz          {{\mathcal Z}}
\def\cb            {\beta}       
\def\cB            {\gamma}      
\def\Cb            {{\ensuremath{{\mathcal C}^\text{rev}_{}}}}
\def\cir           {\,{\circ}\,}
\def\coendint      {\int^\bullet}
 
\def\coevl         {\mathrm{coev}^{\rm l}}
\def\coevr         {\mathrm{coev}^{\rm r}}
\def\Colon         {:\quad}
\def\coNat         {\mathrm{coNat}}
\def\dsty          {\displaystyle }
\def\deltahat      {\hat\delta}
\def\ee            {\end{equation}}
\def\eear          {\end{array}}
\def\End           {\mathrm{End}}
\def\EndC          {\mathrm{End}_\calc}
\def\Enumerate     {\def\leftmargini{1.34em}~\\[-1.42em]\begin{enumerate}}
\def\eq            {\,{=}\,}

\def\evk           {\underline{\mathrm{ev}}{}^{\ko}}
\def\evl           {\mathrm{ev}^{\rm l}}
\def\evr           {\mathrm{ev}^{\rm r}}
\def\FC            {\mathrm F_{\!\calc}^{}}
\def\FM            {\mathrm F_{\!\calm}}
\def\Fun           {{\mathcal Fun}}
\def\Funre         {{\mathcal Rex}}
\def\Funle         {{\mathcal Lex}}
\def\GC            {{\ensuremath{\Xi_\C}}}

\def\hmuver        {\widehat\mu_{\rm ver}}
\def\Hom           {\mathrm{Hom}}
\def\HomC          {\ensuremath{\Hom_\calc}}
\def\HomM          {\ensuremath{\Hom_\calm}}

\def\icHom         {\underline{\mathrm{co}\Hom}}
\def\icoev         {\underline{\mathrm{coev}}}
\def\id            {{\mathrm{id}}}
\def\Id            {{\mathrm{Id}}}
\def\IdC           {\mathrm{Id}_\calc}
\def\IdM           {\mathrm{Id}_\calm}
\def\iHom          {\underline{\Hom}}
\def\iHomM         {\underline{\Hom}_\calm}
\def\iev           {\underline{\mathrm{ev}}}

\def\imuhor        {\underline\mu{}_{\rm hor}}
\def\imuver        {\underline\mu{}_{\rm ver}}
\def\iN            {\,{\in}\,}
\def\endint    {\int_\bullet}
\def\itr           {\underline{\tr}}
\def\izc           {\reflectbox{$\izm$}}
\def\izm           {\zeta}
\def\Itemize       {\def\leftmargini{1.05em}~\\[-1.42em]\begin{itemize}}
\def\Itemizeiii    {\def\leftmargini{2.14em}~\\[-1.52em]\begin{itemize}
                   \addtolength\itemsep{-6pt}}
\def\jc            {\widehat\jmath}
\def\ko            {{\ensuremath{\Bbbk}}}

\def\M             {{\ensuremath\calm}}
\def\muhor         {\mu_{\text{hor}}}
\def\muiHom        {\underline\mu}
\def\muver         {\mu_{\rm ver}}
\def\N             {{\ensuremath\caln}}
\def\Nat           {\mathrm{Nat}}
\def\Nl            {{\rm N}^{\rm l}}
\def\Nr            {{\rm N}^{\rm r}}
\def\NrR           {{\rm S}^{\rm r}_{\Funre}}

\newcommand\Nxl[1] {\\[-1.3em]\\[#1mm]}
\def\Ol            {}  
\def\one           {{\bf1}}

\def\opp           {^\text{opp}}
\def\oti           {\,{\otimes}\,}
\def\otik          {\,{\otimes_\ko}\,}

\def\Phile         {\Phi^{\rm l}}
\def\Phire         {\Phi^{\rm r}}
\def\Psile         {\Psi^{\rm l}}
\def\Psire         {\Psi^{\rm r}}

\def\ra            {^{\rm r.a.}}
\newcommand\rarr[1]{\xrightarrow{~#1~}}
\newcommand\Rarr[1]{\,{\xrightarrow{\,#1\,}}\,}
\def\relcNat       {\underline{\coNat}}
\def\relcNatend    {\underline{\coNat}'}
\def\relid         {\underline{\id}}
\def\relNat        {\underline{\Nat}}
\def\relNatend     {\underline{\Nat}'}
\def\rhohat        {\hat\rho}
\def\Sl            {{\rm S}^{\rm l}}
\def\Sr            {{\rm S}^{\rm r}}

\def\tildepi       {\pi}
\def\Times         {\,{\times}\,}
\def\tmuhor        {\widetilde\muhor}
\def\tmuver        {\widetilde\muver}
\def\To            {\,{\to}\,}
\def\tr            {\mathrm{tr}}

\def\vect          {\ensuremath{\mathrm{vect}}}
\def\Vee           {{}^{\vee\!}}

\newcommand\xarr[1]{\xrightarrow{~#1\,}}

\def\ZA            {{\ensuremath{\calz(\cala)}}}
\def\Zc            {\reflectbox{$\Zm$}}
\def\ZC            {{\ensuremath{\calz(\calc)}}}
\def\ZCM           {\calz(\calcm)}

\def\Zm            {{\mathrm Z}}

\newtheorem{thm}{Theorem}
\newtheorem{cor}[thm]{Corollary}
\newtheorem{lem}[thm]{Lemma}
\newtheorem{prop}[thm]{Proposition}

\theoremstyle{definition}
\newtheorem{Example}[thm]{Example}
\newtheorem{defi}[thm]{Definition}
\newtheorem{Definition}[thm]{Definition}
\newtheorem{rem}[thm]{Remark}

\newtheorem{Lemma}[thm]{Lemma}

\begin{document}

 \numberwithin{equation}{section}

\thispagestyle{empty}
\begin{flushright}
   {\sf ZMP-HH/20-16}\\
   {\sf Hamburger$\;$Beitr\"age$\;$zur$\;$Mathematik$\;$Nr.$\;$865}\\[2mm] August 2020
\end{flushright}

\vskip 2.0em

\begin{center}
{\bf \Large Internal natural transformations}
\\[7pt]
{\bf \Large and Frobenius algebras in the Drinfeld center}

\vskip 18mm

{\large \  \ J\"urgen Fuchs\,$^{\,a} \quad$ and $\quad$ Christoph Schweigert\,$^{\,b}$ }

\vskip 12mm

 \it$^a$
 Teoretisk fysik, \ Karlstads Universitet\\
 Universitetsgatan 21, \ S\,--\,651\,88\, Karlstad
 \\[9pt]
 \it$^b$
 Fachbereich Mathematik, \ Universit\"at Hamburg\\
 Bereich Algebra und Zahlentheorie\\
 Bundesstra\ss e 55, \ D\,--\,20\,146\, Hamburg

\end{center}

\vskip 3.2em

\noindent{\sc Abstract}\\[3pt]
For \M\ and \N\ finite module categories over a finite tensor category \C,
the category $\Funre_\C(\M,\N)$ of right exact module functors is a finite module
category over the Drinfeld center \ZC. We study the internal Homs of this module
category, which we call internal natural transformations. With the help of 
certain integration functors that map \C-\C-bimodule functors to objects of \ZC,
we express them as ends over internal Homs and define horizontal and vertical 
compositions. We show that if \M\ and \N\ are exact \C-modules and \C\ is pivotal,
then the \ZC-module $\Funre_\C(\M,\N)$ is exact. We compute its
relative Serre functor and show that if \M\ and \N\ are even
pivotal module categories, then $\Funre_\C(\M,\N)$ is pivotal as well.
Its internal Ends are then a rich source for Frobenius algebras in \ZC.



\newpage

\section{Introduction}

Module categories over monoidal categories have been a prominent topic in
representation theory in the past two decades. The theory is particularly
well-developed for finite tensor categories and their finite module and bimodule
categories. Indeed, many notions and results in the theory of finite-dimensional
representations over finite-dimensional Hopf algebras have found their natural
conceptual home in this setting. Examples of such notions include the unimodularity
of a finite tensor category and factorizability of a braided finite tensor
category. Results include Radford's $S^4$-formula \cite{etno2}, including its
generalization to bimodule categories \cite{fScS2}, the equivalence of various
characterizations of the non-degeneracy of a braiding on a finite tensor category
\cite{shimi10}, and the theory of `reflections' of Hopf algebras \cite{balS}. Moreover,
module and bimodule categories have been used intensively in the study of subfactors, of
two-dimensional conformal field theory, and of three-dimensional topological field theory.

The following fact about module categories is well known. Let \C\ be a finite tensor
category and $\M$ and $\N$ be finite \C-modules. Then the category $\Funre_\C(\M,\N)$
of right exact module functors is a finite module category over the Drinfeld center
$\ZC$ (which is a finite tensor category). In this paper we study the internal Homs
$\iHom_\ZC(G,H)$ for $G,H \iN \Funre_\C(\M,\N)$. We denote these 
internal Homs by $\relNat(G,H) \iN\ZC$ and call them \emph{internal natural transformations}.

For the vector space of ordinary natural transformation between two linear functors, 
the Yoneda lemma implies a useful formula in terms of an end over morphism spaces:
  \be
  \Nat(G,H) \,= \int_{\!m\in\M}\! \Hom_\N(G(m),H(m)) \,.
  \ee
The structure morphisms $\Nat(G,H) \To \Hom_\N(G(m),H(m))$
of this end just give the components of the natural transformation.
One of the main results of this paper, Theorem \ref{thm:main1}, is a similar expression 
  \be
  \relNat(F,G) \,= \int_{\!m\in\M} \iHom_\N(F(m),G(m))
  \ee
for the internal natural transformations as objects in \ZC. In particular, we show that 
the end on the right hand side has a natural structure of an object in the Drinfeld center. 

A crucial ingredient that allows us to obtain this result are two functors
  \be
  \mbox{$\endint$}: ~~ \Fun_{\C|\C} ({}^{\#\!\!}\calm \boti \calm, \C) \to \ZC
  \qquad\text{and}\qquad
  \mbox{$\coendint$}: ~~ \Fun_{\C|\C} (\calm^\#\boti\calm, \C) \to \ZC 
  \ee
given by
  \be
  \mbox{$\endint$}:~~ G \,\mapsto \int_{\!m\in\calm}\! G(m,m) \qquad\text{and}\qquad
  \mbox{$\coendint$}:~~ H \,\mapsto \int^{m\in\calm}\!\! H(m,m)
  \label{eq:intro:intfu}
  \ee
for $G \iN \Fun_{\C|\C} ({}^{\#\!\!}\calm \boti \calm, \C)$ and
$H \iN \Fun_{\C|\C} (\calm^\#\boti\calm, \C)$, respectively, where ${}^{\#\!\!}\calm$
and $\calm^\#$ are two right \C-module structures on the opposite category $\calmopp$.
The existence of these functors, which we call \emph{central integration functors},
is shown in Theorem \ref{thm:integralprop}.

Since the internal natural transformations are internal Homs,
they come 
with associative compositions. It follows in particular that for any module functor $F$ the 
object $\relNat(F,F)$ has a natural structure of a unital associative algebra in \ZC. We 
show that the structure morphisms $\relNat(G,H) \To \iHom_\N(G(m),H(m))$
of the end behave in the same way as the component maps of an ordinary 
natural transformation.
This allows us to define horizontal and vertical compositions which obey
the Eckmann-Hilton relation. As a consequence, the object $\relNat(\Id_\M,\Id_\M)$ 
of internal natural endotransformations of the identity functor is a commutative algebra
in the braided category \ZC.

We also study the situation that the monoidal category \C\ has the additional
structure of a \emph{pivotal} tensor category. (This endows its Drinfeld
center \ZC\ with a pivotal structure as well.)  Moreover, we assume that
the module categories under investigation are now \emph{exact} \C-modules.
As we show in Proposition \ref{prop:intpivotal}, in this case the two central integration 
functors \eqref{eq:intro:intfu} are related by the Nakayama functor $\Nr_\M \iN \Funre(\M,\M)$
according to
  \be
  \mbox{$\coendint$} = \mbox{$\endint$} \,\circ \big( \Id_{\calmopp}\boti \Nr_\M\big) \,.
  \ee
Based on this result we show in Theorem \ref{thm:exactovercenter} that for any pair
$\M_1,\M_2$ of exact module categories over a pivotal finite tensor category \C, the category
$\Funre_\C(\M_1,\M_2)$ is an \emph{exact} \ZC-module. Specifically,
we compute its relative Serre functor to be 
  \be
  \NrR = \Nr_{\N} \circ (D \act -) \circ \Nr_{\M} \,,
  \ee
with $D$ the distinguished invertible object of \C. In Corollary \ref{cor:pivotalovercenter} 
we then conclude that in case \C\ is unimodular, this exact module category is \emph{pivotal} (in the sense of Definition \ref{def:Mpivotal}).
It follows that in this case $\relNat(F,F)$ has the structure of
a Frobenius algebra, and in particular $\relNat(\Id_\M,\Id_\M)$ has the structure of a
commutative Frobenius algebra.
In this way, \C-module categories become a rich source of Frobenius algebras in \ZC.

\medskip

This paper is organized as follows. After setting the stage in Section \ref{sec:Background},
in Section \ref{sec:cifunctors} we study relations between bimodule functors with codomain
\C\ and the Drinfeld center of \C, which leads us to the notion of central integration functors.
Section \ref{sec:inat} deals with internal natural transformations.
In particular, in Section \ref{sec:inat-end} we explain how they can be expressed as
an end, and in Section \ref{sec:inat-compositions} we introduce and study their
horizontal and vertical compositions. Finally, in Section \ref{sec:e+p}
we combine these results with the theory of relative Serre functors and
pivotal module categories to examine exactness and pivotality of the
the functor category $\Funre_\C(\calm,\caln)$ as a module category over \ZC.

\medskip

A direct application of our results (and, in fact, also a major motivation for our
investigations) is in the description of bulk fields in rigid logarithmic two-dimensional
conformal field theories, i.e.\ conformal field theories whose chiral data are
described by a modular finite tensor category \C. This application
will be discussed in detail elsewhere. Here we content ourselves with
mentioning the basic idea. When \C\ is modular, then we have a braided equivalence
$\ZC \,{\simeq}\, \C^\text{rev}_{}\boti\C$. The algebra of bulk fields (or, more
generally, disorder and defect fields) in full local conformal field theory
can therefore be regarded as an object in \ZC.
 
The field algebras in local conformal field theories should be Frobenius algebras; this 
has e.g.\ been demonstrated for bulk algebras of rigid logarithmic conformal field theories 
in \cite{fuSc22}. It is also well known that there are different full local conformal field 
theories that share the same chiral data based on a given modular tensor category \C.
It has been established almost two decades ago that in case \C\ is semisimple, the datum that 
in addition to the chiral data is needed
to characterize a local conformal field theory is a (semisimple, indecomposable) \C-module 
category \cite{fffs3,fjfrs}. The results of Section \ref{sec:e+p} show that for \C\ 
not semisimple, a \emph{pivotal} indecomposable module category \M\ is a natural
candidate for such an additional datum.  Boundary conditions of the full conformal field
theory are then described by objects $m\iN \M$ and boundary fields by internal Homs
$\iHom(m,m')\iN\C$. By Theorem 3.15 of \cite{shimi20}, the algebra
$\iHom(m,m)$ is a symmetric Frobenius algebra for any $m\iN\M$. A right exact
module functor $G\iN \Funre_\C(\M_1,\M_2)$ describes a topological
defect line between local conformal field theories characterized by $\M_1$ and
by $\M_2$, respectively. It is then natural to propose that the defect fields that
change a defect line labeled by $G$ to a defect line labeled by $H$ are
given by $\relNat(G,H)\iN \ZC$. In particular, $\relNat(G,G)\in\ZC$
is a symmetric Frobenius algebra; as a special case, $\relNat(\Id_M,\Id_\M)\iN\ZC$ 
a commutative symmetric Frobenius algebra, as befits the space of bulk fields.
This proposal also leads to natural candidates for operator product expansions 
and passes non-trivial consistency checks.


\section{Background}\label{sec:Background}

In this section we fix our notation and mention some pertinent structures and concepts.

\subsection{Basic concepts}

\paragraph{Monoidal categories.}

We denote the tensor product of a monoidal category by $\otimes$ and the monoidal
unit by $\one$, and the associativity and unit constraints by $\alpha$, $l$ and $r$,
i.e.\ a monoidal category is a quintuple
$\calc \eq (\calc,\otimes,\one,\alpha,l,r)$. For better readability of various formulas, we
sometimes take, without loss of generality, the tensor product to be strict, i.e.\ take the 
associator $\alpha$ and unit constraints $l$ and $r$ to be identities.


\paragraph{Module categories.}

The notion of a (left) \emph{module category} $\calm$ over a monoidal category \C, or
\C-module, for short, categorifies the notion of module over a ring: There is an 
\emph{action functor} $\C \Times \M \To \M$, exact in its first variable, together with
a mixed associator and a mixed unitor that obey mixed pentagon and triangle relations.
For background on module categories, as well as module functors and module
natural transformations, see e.g.\ \Cite{Ch.\,7}{EGno} or \Cite{Sect.\,2.3}{shimi17}.
In the present paper, module categories will be left modules unless stated otherwise.
We denote the action morphism by a dot and the mixed associator by $a$, i.e.\ $a$ has components $a_{c,c',m}\colon (c\oti c')\act m \To c \act (c'\act m)$
with $c,c'\iN\C$ and $m\iN\calm$.
The natural isomorphism that defines the structure of a \C-module functor $G$ is denoted
by $\phi^G$, i.e.\ $\phi^G$ has components $\phi^G_{c,m}\colon G(c\Act m) \To c \act (G(m))$
with $c\iN\C$ and $m\iN\calm$.


\paragraph{Finite categories.}

We fix an algebraically closed field \ko. A \emph{finite} \ko-linear category 
is an abelian category that is equivalent as abelian category to the category of
finite-dimensional modules over a finite-dimensional \ko-algebra.
A \emph{finite tensor category} is a finite \ko-linear category which is rigid monoidal
with appropriate compatibility conditions among the structures, see e.g.\ \cite{etos} 
or \Cite{Sect.\,2.5}{shimi17}. Since a finite tensor
category is rigid, its tensor product functor is exact.
Our conventions concerning dualities of a rigid category \C\ are as follows. 
The right dual of an object $c$ is denoted by $c^\vee$, and the right evaluation 
and coevaluation are morphisms
  \be
  \evr_c \in \Hom_\calc (c^\vee \oti c,\one) \qquad{\rm and}\qquad
  \coevr_c \in \Hom_\calc (\one ,c \oti c^\vee) \,,
  \ee
while the left evaluation and coevaluation are
  \be
  \evl_c \in \Hom_\calc (c \oti \Vee c,\one) \qquad{\rm and}\qquad
  \coevl_c \in \Hom_\calc( \one ,\Vee c \oti c)
  \ee
with $\Vee c$ the left dual of $c$.

A module category $\calm$ over a finite tensor category \C\ is called \emph{finite} iff
$\calm$ is a finite \ko-linear abelian category and the action of \C\ on $\calm$ is
linear and right exact in both variables. 
The category of right exact module endofunctors of a finite module category $\calm$ over 
a finite tensor category \C\ is again a finite tensor category \Cite{Prop.\,7.11.6.}{EGno};
we denote it by $\calcm$.
A finite \C-module is called \emph{exact} iff $p\act m$ is projective in $\calm$
for each projective $p\iN\calc$ and each $m\iN\calm$.
In particular, $\C$ is an exact module category over itself \Cite{Def.\,3.1}{etos}.
Indecomposable exact module categories over $H$-mod, for $H$ a
finite-dimensional Hopf algebra, are classified in \Cite{Sect.\,3.2}{anMo}.
For recent results see also \cite{shimi20}.


\paragraph{Drinfeld centers.}

For $\cala$ a monoidal category, a \emph{half-braiding} for an object $a_0\iN\cala$ 
is a natural family $\sigma \eq (\sigma_a)_{a\in \cala}$ of morphisms $\sigma_a
\colon a \oti a_0\To a_0 \oti a$ such that (suppressing the associator of $\cala$)
$\sigma_{a\otimes a'} \eq (\sigma_a \oti \id_{a'}) \cir (\id_a \oti \sigma_{a'})$
for all $a,a' \iN \cala$ and $\sigma_\one \eq \id_{a_0}$. 
The \emph{Drinfeld center} $\calz(\cala)$ of $\cala$ has as objects
pairs $(a,\sigma)$ consisting of an object of $\cala$ and a half-braiding on it.
The morphisms $\Hom_{\ZA}((a,\sigma),(a',\sigma'))$ are those morphisms
$a\,{\xrightarrow{~f\,}}\, a'$ in $\cala$ that satisfy
$(f \oti \id_b) \cir \sigma_b^{} \eq \sigma'_b \cir (\id_b \oti f)$ for all $b\iN\cala$. 
For any monoidal category $\cala$, the Drinfeld center $\calz(\cala)$ 
has a natural braided monoidal structure.


\paragraph{Unimodular categories.}

In any finite tensor category there is (uniquely up to isomorphism) a
\emph{distinguished invertible object} $D$,
an invertible object that comes \Cite{Thm\,3.3}{etno2} with
coherent isomorphisms $D \oti x \,{\cong}\, x^{\vee\vee\vee\vee} {\otimes}\, D$.
A \emph{unimodular} finite tensor category is a finite tensor category $\cala$ for which
the distinguished invertible object is the monoidal unit.
There are several equivalent characterizations of unimodularity \cite{shimi7}, e.g.\
the forgetful functor $U\colon \calz(\cala)\,{\to}\,\cala$ from the Drinfeld center
is a Frobenius functor.
    

\paragraph{Modular categories.}

For \C\ a braided finite tensor category, we denote by \Cb\ its \emph{reverse}, i.e.\ the
same monoidal category, but with inverse braiding. There is a canonical braided functor
  \be
  \GC: \quad \Cb \boti \C \to \ZC
  \label{eq:CbC->ZC}
  \ee
from the \emph{enveloping category} of \C, i.e.\ the Deligne product of \Cb\ with \C\
(which exists, as \C\ is finite abelian), to the Drinfeld center of \C.
As a functor, \GC\ maps the object $\Ol u \boti v \iN \Cb \boti \C$ to the tensor product
$\Ol u \oti v \iN \C$ endowed with the half-braiding $\cB_{\Ol u\otimes v}$ that has components
$ \cB_{c;\Ol u\otimes v} \,{:=}\,
  ( \id_{\Ol u} \oti \cb_{v,c}^{-1} ) \cir (\cb_{c,\Ol u} \oti \id_v )$
for $c \iN \C$, with $\cb$ the braiding in \C.
The braided monoidal structure on the functor \GC\ is given by the coherent
family $\id_{\Ol u} \oti \cb_{v,\Ol x} \oti \id_y$ of isomorphisms from
$\Ol u \oti v \oti \Ol x \oti y$ to $\Ol u \oti \Ol x \oti v \oti y$.

A finite tensor category \C\ is \emph{non-degenerate} iff the functor $\GC$ is an
equivalence. If $\C$ is even a ribbon category, then \Cb\ is a ribbon category with
the inverse twist. A non-degenerate finite
ribbon category is a \emph{modular tensor category}, or modular category, for short.
Traditionally, the term modular category has been used under the additional
assumption that the finite tensor category \C\ is semisimple, i.e.\ a
fusion category; in our context, such a restriction is not natural.
A modular category is in particular unimodular.


\paragraph{The central monad and comonad.}

The Drinfeld center comes with a forgetful functor $U\colon \ZC\,{\to}\,\calc$ that omits the
half-braiding. $U$ is exact and hence has a left and a right
adjoint. These adjunctions are (co)monadic and thus give rise to a monad
  \be 
  \begin{array}{rl}
  \Zm \Colon \calc & \!\! \xrightarrow{~~~\,} \,\calc \\
  c & \!\! \longmapsto \,\int^{x\in\calc}\! x \oti c\oti \Vee x
  \eear
  \ee
on $\calc$ and to a comonad
  \be
  \begin{array}{rl}
  \Zc \Colon \calc & \!\! \xrightarrow{~~~\,} \,\calc \\
  c & \!\! \longmapsto \,\int_{x\in\calc} \Vee x \oti c \oti x
  \eear
  \label{ccomonad}
  \ee
(or, equivalently, $ c \,{\mapsto}\,\int_{x\in\calc} x \oti c \oti x^\vee$).
These are called the \emph{central monad} and \emph{central comonad}, respectively.
Since the adjunctions are monadic and comonadic respectively, we have canonical
equivalences of the categories of $\Zm$-modules, $\Zc$-comodules and the Drinfeld center $\ZC$.  
For more details about the central monad and comonad see e.g.\ \cite{brVi5} and
\Cite{Ch.\,9}{TUvi}.

When dealing with coends, such as the one defining the central monad, a convenient tool is 
the following form of the Yoneda lemma (see e.g.\ \Cite{Prop.\,2.7}{fScS2}):

\begin{Lemma}\label{lem:yoncoend}
Let $\cala$ and $\calb$ be finite linear categories and $F\colon \cala \To \calb$ 
a linear functor. Then there is a natural isomorphism 
  \be
  \int^{a \in \cala}\!\! \Hom_\cala( a,-) \otimes F(a) \,\cong F 
  \label{eq:delta-property}
  \ee
of linear functors. 
\end{Lemma}

An analogous co-Yoneda lemma holds for ends:
$\int_{a\in\cala}\! \Hom_\cala(-,a)^* \oti F(a) \,{\cong}\, F$. The isomorphisms 
in these formulas are uniquely determined by universal properties; accordingly,
from now on we will write them as equalities.


\paragraph{Eilenberg-Watts calculus.}
For any pair of finite linear categories $\cala$ and $\calb$ there are two pairs of
two-sided adjoint equivalences 
  \be
  \begin{tikzcd}
  \Funle(\cala,\calb) \ar[yshift=3pt]{r}{\Psile}
  & \cala\opp \boti \calb \ar[yshift=-3pt]{l}{\Phile} \ar[yshift=-3pt]{r}[swap]{\Phire}
  & \Funre(\cala,\calb) \ar[yshift=3pt]{l}[swap]{\Psire}
  \end{tikzcd}
  \label{eq:EW}
  \ee
given by \cite{shimi7,fScS2}
  \be
  \begin{array}{lll}
  & \Phile(\overline a \boti b) := \Hom_\cala(a,-) \oti b \,, \quad
  & \Psile(F) := \int^{a\in\cala} \overline a \boti F(a)
  \Nxl3
  \text{and}~~ & \Phire(\overline a \boti b) := {\Hom_\cala(-,a)}^{\!*}_{} \oti b \,,
  & \Psire(G) := \int_{a\in\cala} \overline a \boti G(a) \,.
  \eear
  \label{eq:def:PhiPsi}
  \ee
This provides a Morita invariant version of the classical Eilenberg-Watts description of 
right or left exact functors between the categories $R$-mod and $S$-mod of modules over
unital rings in terms of $S$-$R$-bimodules. The functors \eqref{eq:def:PhiPsi}
are therefore called Eilenberg-Watts equivalences.

Applying these equivalences to the identity functor on $\cala$, regarded as a left exact 
functor, yields a right exact endofunctor
  \be
  \Nr_\Cala := \Phire\cir\Psile(\Id_\cala) = \int^{a\in\cala}\!\! \Hom_\cala(-,a)^*_{} \oti a 
  ~ \in \Funre(\cala,\cala) \,,
  \label{eq:def:Nr}
  \ee
which is called the \emph{Nakayama functor} of the finite linear category $\cala$
\Cite{Def.\,3.14}{fScS2}. Analogously, by applying $\Phile\cir\Psire$ 
to $\Id_\cala$ regarded as a right exact functor we obtain a left exact analogue
$\Nl_\Cala \eq \int_{a\in\cala} \Hom_\cala(a,-) \oti a \iN \Funle(\cala,\cala)$.
The functor $\Nl_\Cala$ is left adjoint to $\Nr_\Cala$.
For $\cala$ and $\calb$ finite tensor categories, the Nakayama functor of an 
$\cala$-$\calb$-bi\-mo\-dule \M\ has a natural structure of a \emph{twisted} 
bimodule functor, in the sense that there are coherent isomorphisms
  \be
  \Nr_\Calm(a.m.b) \cong {}^{\vee\vee\!}a.\,\Nr_\Calm(m)\,.b^{\vee\vee}
  \ee
for all $m\iN\calm$, $a\iN\cala$ and $b\iN\calb$ \Cite{Thm.\,4.5}{fScS2}.


\subsection{Internal Hom}

For any $m\iN\calm$ we denote the \emph{action functor} by $H_m \,{=}\, {-}\Act m\colon 
\,\calc\To\calm$. As $H_m$ is (right) exact, the following functors exist:
	  
\begin{Definition}
Let \C\ be a monoidal category and $\calm$ be a
\C-module. An \emph{internal Hom} of $\calm$ in \C\ is a functor 
  \be
  \iHomM(\reflectbox?;?)\Colon \calmopp\boti\calm \to\C
  \ee
such that for every $m\iN\calm$ the functor $\iHomM(m,-)\colon \calm\To\calc$ is 
right adjoint to the action functor $H_m$, i.e.\ such that
for any two objects $m,m' \iN \calm$ there is a natural family 
  \be
  \HomC(c,\iHomM(m,m')) \cong \HomM(c\act m,m') 
  \label{eq:def:iHom}
  \ee
of isomorphisms (see e.g.\ \cite{ostr}).
\end{Definition}
 
Being a right adjoint, the internal Hom is left exact.
When it is clear from the context which module category \M\ is concerned, we simply write
$\iHom$ in place of $\iHomM$.

We note that there is no separate notion of a `relative adjoint' functor:

\begin{lem}\label{lem:nixreladjoint}
Let $\calc$ be a monoidal category and $\calm$, $\caln$ be
$\calc$-modules. Let $F\colon \calm\To\caln$ and $G\colon \caln\To\calm$ be module functors.
Then $F$ and $G$ are adjoint functors if and only if there are functorial isomorphisms
  \be
  \iHom_\caln(F(m),n) \,\cong\, \iHom_\calm(m,G(n)) \,.
  \label{eq:inneradj}
  \ee
\end{lem}

\begin{proof}
By the definition of the internal Hom and the fact that $F$ is a module functor, we have
  \be
  \bearll
  \Hom_\calm(c\act m, G(n)) \cong \Hom_\calc\big(c,\iHom_\calm(m,G(n))\big) \qquad\text{and}
  \Nxl3
  \Hom_\caln(F(c. m),n) \cong
  \Hom_\caln(c\act F(m),n) \cong \Hom_\calc\big(c,\iHom_\caln(F(m),n) \big)
  \eear
  \ee
for every $c\iN \calc$, $m\iN\calm$ and $n\iN\caln$. Thus \eqref{eq:inneradj} implies
that $F$ and $G$ are adjoint. The converse holds by the Yoneda lemma.
\end{proof}


\paragraph{Algebra structure on internal Ends.}

The counits of the adjunctions $H_m \,{\dashv}\, \iHom(m,-)$
provide for any two objects $m,m'\iN\calm$ a canonical morphism
  \be
  \iev_{m,m'} : \quad \iHom(m,m') \act m \to m' ,
  \label{eq:def:iev}
  \ee
in $\calm$ given by the image of the identity morphism in $\EndC(\iHom(m,m'))$ under the
defining isomorphism \eqref{eq:def:iHom}. The composition
$\iev_{m'.m''} \cir (\id_{\iHom(m',m'')} \oti \iev_{m.m'}) \cir a_{\iHom(m',m''),\iHom(m,m'),m}$
furnishes an associative multiplication
  \be
  \muiHom_{m,m',m''} :\quad \iHom(m',m'') \otimes \iHom(m,m') \,\xarr{~}\, \iHom(m,m'')
  \label{eq:mult}
  \ee
on internal Homs. Moreover, the component 
  \be
  \one \xrightarrow~ \iHom(m,\one\Act m) = \iHom(m,m)
  \label{eq:unit4muiHom}
  \ee
of the unit of the adjunction $H_m \,{\dashv}\, \iHom(m,-)$ 
is a unit for the multiplication \eqref{eq:mult}.


\paragraph{Compatibility with module functors.}

A module functor $G\colon \calm\To\caln$ induces a natural morphism
$\HomM(c.m,m') \To \Hom_\caln(c.G(m),G(m'))$ for any $m,m'\iN\calm$ and $c\iN\calc$,
and thus $\HomC(c,\iHom(m,m')) \To \HomC(c,\iHom(G(m),G(m'))$.
By the Yoneda lemma this induces a morphism 
  \be
  \underline G\Colon \iHom(m,m') \rarr{} \iHom(G(m),G(m'))
  \label{underlineG}
  \ee
of internal Hom objects in \C.
It is easy to check that $\underline{G\,{\circ}\, G'} \eq \underline G \,{\circ}\,\underline {G'}$,
where on the left hand side one deals with composition of functors and on the right hand side
with composition of morphisms in $\calc$. 
	 


\paragraph{Internal Hom as a bimodule functor.}

For a left \C-module $\calm$ the opposite category $\calmopp$ can be endowed in many ways
with the structure of a right \C-module, which are related by the monoidal functor of 
taking biduals. Of relevance to us are the following two choices of the right \C-action: either 
$\overline m\act c \,{:=}\, \overline{\Vee c\act m}$ for $m \iN \calm$, or else 
$\overline m\act c \,{:=}\, \overline{ c^\vee.\,m}$. We denote the former right \C-module by
${}^{\#\!\!}\calm$ and the latter by $\calm^\#$. Then in particular both
${}^{\#\!\!}\calm \boti \calm$ and $\calm^\# \boti \calm$ have a natural structure of a
\C-bimodule. It follows that $\iHom$ is naturally a bimodule functor,
with \C\ regarded as a bimodule over itself;

\begin{lem}{\rm \Cite{Lemma\,2.7}{shimi20}} \label{lem:bimodstr}
The functor $\iHom(\reflectbox{$?$};?)\colon {}^{\#\!\!}\calm \boti \calm \To\C $
is a bimodule functor.
\end{lem}

\begin{proof}
For any $c \iN \C$ the endofunctor $F_c$ of \M\ defined by $F_c(m) \eq c\act m$ is left exact
and thus has a left adjoint.
Indeed we have 
  \be
  \bearll
  \HomC(\gamma,\iHom(m,c.m')) \!\! &
  \cong \HomM(\gamma\act m,c\act m') \,\cong\, \HomM(c^\vee.\,(\gamma. m),m')
  \Nxl3 &
  \cong \HomC(c^\vee\oti \gamma,\iHom(m,m')) \,\cong\, \HomC(\gamma,c\act\iHom(m,m'))
  \eear
  \ee
for any $\gamma \iN \C$. Similarly, $ \HomC(\gamma,\iHom(c.m,m'))
\,{\cong}\, \HomM(\gamma,\Hom(m,m')\oti c^\vee)$
for $\gamma \iN \C$. Thus there are isomorphisms
  \be
  \iHom(m,c'\act m') \cong c' \oti \iHom(m,m') \qquad\text{and}\qquad
  \iHom(c \act m,m') \cong \Hom(m,m')\oti c^\vee .
  \ee
Taken together, these imply the natural isomorphisms
  \be
  \iHom(c\act m,c'\act m') \cong c' \oti \iHom(m,m') \oti c^\vee
  \label{eq:ostr3.3}
  \ee
that are required for $\iHom(\reflectbox{$?$};?)$ to be a bimodule functor.
\end{proof}


\paragraph{Internal coHom.}
There is an obvious dual notion to the internal Hom: For $\C$ a monoidal category and
$\calm$ a $\C$-module, an \emph{internal coHom} of $\calm$ in $\C$ is a functor 
$\icHom(\reflectbox?;?)\colon \calmopp\boti\calm 
     $\linebreak[0]$ {\to}\,
\C$
such that for any $m,m' \iN \calm$ there is a natural family 
  \be
  \HomC(\icHom(m',m),c) \cong \HomM(m,c\act m') 
  \label{eq:def:icHom}
  \ee
of isomorphisms. By exactness of the action functors $H_m$, also internal coHoms exist.
Being a left adjoint, the internal coHom is right exact.

On the left hand side of \eqref{eq:def:icHom} we have $ \HomC(\icHom(m',m),c)
\,{\cong}\,\HomC(c^\vee,\icHom(m',m)^\vee)$, while the right hand side is
$\HomM(m,c\act m') \,{\cong}\, \HomC(c^\vee,\iHom(m,m'))$. Thus
the internal Hom and coHom are are indeed dual to each other: 
  \be
  \icHom(m',m)^\vee \cong\iHom(m,m') \qquad\text{and}\qquad
  \icHom(m',m)\cong \Vee\iHom(m,m') \,.
  \label{eq:icHom-iHom}
  \ee
By taking left duals in \eqref{eq:mult} we obtain a coassociative comultiplication
  \be
  \bearll
  \icHom(m'',m) = \Vee\iHom(m,m'') \!\! & \xarr{~} \Vee\big(\iHom(m',m'') \otimes \iHom(m,m') \big)
  \Nxl3 & \hspace*{4.5em}
  \cong\, \icHom(m',m) \otimes \icHom(m'',m') \,.
  \eear
  \label{eq:comult}
  \ee
A counit for this comultiplication is given by the left dual of the unit
for the multiplication \eqref{eq:mult}

Analogously to the morphism \eqref{underlineG},
for any module functor $G\colon \calm\To\caln$ 
we get a morphism $\underline G \colon \icHom(m,m') \,{\rarr{}}\, \icHom(G(m),G(m'))$
of internal coHom objects in \C.
And analogously to Lemma \ref{lem:bimodstr} one shows that
$\icHom(\reflectbox{$?$};?)\colon \calm^\#\boti\calm \To\C $ is naturally a bimodule functor.

\begin{rem} \label{rem:rem4}
For \C\ as a module over itself, the internal Hom is $\iHom(c,c') \eq c' \oti c^\vee$, and
the family \eqref{eq:def:iHom} reduces to the natural isomorphism that is furnished by the
right duality.  Similarly we then have $\icHom(c,c') \eq c'\oti \Vee c$.   
\end{rem}


\subsection{Pivotal module categories}

An additional structure that a finite tensor category \C\ may admit is a \emph{pivotal} 
structure, i.e.\ a monoidal isomorphism $\pi\colon \Id_\C \To -^{\!\vee\vee}$
from the identity functor to the right double-dual functor. The presence of a pivotal
structure has important consequences; for instance, while the notion of a Frobenius 
algebra makes sense in any monoidal category \C, the one of a symmetric Frobenius 
algebra does so only if \C\ is pivotal (and even depends on the choice of pivotal 
structure). Pivotality will be used extensively in Section \ref{sec:e+p}; we therefore
discuss it here in some detail.

If \M\ and \N\ are left modules over a pivotal finite tensor category \C, then
the Eilenberg-Watts equivalences \eqref{eq:EW} of linear categories induce adjoint equivalences
involving categories of left and right exact module functors: we have \Cite{Prop.\,4.1}{fScS3}
  \be
  \begin{tikzcd}
  \Funle_\calc(\M,\N) \ar[yshift=3pt]{r}{\Psile}
  & \calz(\M^\# \boti \N) \ar[yshift=-3pt]{l}{\Phile} 
  \end{tikzcd}
  \quad~ \text{and}\quad~
  \begin{tikzcd}
  \calz({}^{\#\!\!}\M \boti \N) \ar[yshift=-3pt]{r}[swap]{\Phire}
  & \Funre_\calc(\M,\N) \ar[yshift=3pt]{l}[swap]{\Psire}
  \end{tikzcd} 
  \label{eq:moduleEW0}
  \ee
where $\M^\#$ and ${}^{\#\!\!}\M$ are the right modules with underlying linear category
$\calmopp$ described above.

\begin{rem} \label{rem:ngsc}
Any pivotal tensor category is equivalent, as a pivotal category, to a
strict pivotal category \Cite{Thm.\,2.2}{ngsc}. Thus in case the finite tensor
category \C\ of our interest has a pivotal structure, for many purposes we may
replace it by a strict pivotal one in which $c^\vee \,{=}\, \Vee c$ holds for every $c\iN\C$. 
When doing so, ${}^{\#\!\!}\calm$ and $\calm^\#$ are the same \C-module; we 
denote it by the symbol $\overline\calm$. The equivalences \eqref{eq:moduleEW0} then
combine to
  \be
  \begin{tikzcd}
  \Funle_\C(\M,\N) \ar[yshift=3pt]{r}{\Psile}
  & \calz(\overline\M \boti \N) \ar[yshift=-3pt]{l}{\Phile} \ar[yshift=-3pt]{r}[swap]{\Phire}
  & \Funre_\C(\M,\N) \,. \ar[yshift=3pt]{l}[swap]{\Psire}
  \end{tikzcd}
  \label{eq:moduleEW}
  \ee
Furthermore, in this case the Nakayama functor $\Nr_{\!\M}$ of \M\ becomes an ordinary
module functor, rather than a module functor twisted by a double dual. 
\end{rem}

Next we
note that the Drinfeld center $\ZC$ of a rigid category \C\ is rigid as well.
Moreover, the components $\pi_c\colon c\To c^{\vee\vee}$ 
of a pivotal structure on \C\ are even morphisms in \ZC. Thus if \C\ is pivotal,
then \ZC\ inherits a distinguished pivotal structure \Cite{Exc.\,7.13.6}{EGno}.
For pivotal \C\ we will consider $\ZC$ with this pivotal structure.

\emph{Exact} module categories over a pivotal tensor category turn out to have a
particularly interesting theory. Let us first recall 

\begin{Definition}\Cite{Def.\,4.22}{fScS2} \label{def:4.22fScS2}
Let $\calm$ be a left $\calc$-module. A \emph{right relative Serre functor} on $\calm$
is an en\-do\-functor $\Sr_\calm$ of $\calm$ equipped with a family
  \be
  \iHom(m,n)^{\vee} \xrightarrow{~\cong~} \iHom(n,\Sr_\calm(m))
  \label{eq:Serref}
  \ee
of isomorphisms natural in $m,n\iN\calm$.
A \emph{left relative Serre functor} $\Sl_\calm$ on $\calm$ comes analogously with a family
  \be
  {}^{\vee}\iHom(m,n) \xrightarrow{~\cong~} \iHom(\Sl_\calm(n),m)
  \label{eq:leftSerre}
  \ee
of natural isomorphisms.
\end{Definition}

\begin{rem} \label{rem:S-vs-N}
According to \Cite{Thm.\,4.26}{fScS2} the Nakayama and relative Serre functors of an exact 
module \M\ over a finite tensor category \C\ are related by $\Nl_\M \,{\cong}\, D_\C\act\Sl_\M$
and $\Nr_\M \,{\cong}\, D_\C^{-1}\Act\,\Sr_\M$.
In particular the Nakayama and relative Serre functors coincide iff \C\ is unimodular.
\end{rem}

It is known \Cite{Prop.\,4.24}{fScS2} that a finite left \C-module admits a relative
Serre functor if and only if it is an exact module category. In this case the relative
Serre functor is an equivalence of
categories. A right relative Serre functor on $\calm$ is a twisted module functor
\Cite{Lemma\,4.23}{fScS2} in the sense that there are coherent natural isomorphisms
  \be
  \phi_{c,m}^{\Sr}:\quad \Sr_\calm(c.m) \,\xrightarrow{~\cong~} c^{\vee\vee} \Act\, \Sr_\calm(m) \,.
  \label{eq:Sr(cm)=cvv.Sr(m)}
  \ee
Similarly there are coherent natural isomorphisms
$\Sl_\calm(c.m) \,{\cong}\, {}^{\vee\vee\!}c \act \Sl_{\calm}(m) $.
These results allow one to give

\begin{Definition}\label{def:Mpivotal} ($\!\!$\Cite{Def.\,5.2}{schaum5} and \Cite{Def.\,3.11}{shimi20})
A \emph{pivotal structure}, or \emph{inner-product structure}, on an exact module category 
$\calm$ over a pivotal finite tensor category $(\calc,\pi)$ is an isomorphism
$\tildepi^\calm\colon \Id_\calm \,{\xrightarrow{\,\cong}\,}\, \Sr_\calm$ of functors such that
the equality $\phi^{\Sr}_{c,m} \cir \tildepi^\calm_{c \Act m} \eq \pi_c\act\tildepi^\calm_m$
of morphisms from $c\act m$ to $c^{\vee\vee}\Act\,\Sr_\calm(m)$ holds
for every $c\iN\calc$ and every $m\iN\calm$.
\end{Definition}

In short, a pivotal structure is an isomorphism, as module functors, from the identity functor
to the Serre functor, where the pivotal structure on \C\ has been used to turn them
into module functors of the same type.
If the module category $\calm$ is indecomposable, then the identity functor
$\Id_\calm$ is a simple object in the category of right exact module endofunctors.
Thus Schur's lemma implies

\begin{Lemma}\Cite{Lemma\,3.12}{shimi20}
Let $\calm$ be an indecomposable exact module category over a pivotal finite tensor
category. A pivotal structure on $\calm$, if it exists, is unique up to a scalar multiple.
\end{Lemma}

\begin{rem} \label{rem:rem9}
(i)\, As a special case of Lemma 4.23 of \cite{fScS2}, consider a pivotal 
finite tensor category \C\ as a module category over itself. We have
  \be
  \Sl_\C(c.\one) = {}^{\vee\vee}\!c \act \Sl_\C(\one) = {}^{\vee\vee}\! c \,,
  \ee
so in this case $\Sl_\C$ coincides with the bidual functor. The pivotal structure of 
\C\ then endows the module category ${}_\C\C$ with the structure of a pivotal module category.
 \\[2pt]
(ii)\, It follows \Cite{Thm.\,3.13}{shimi20} that for $\calm$ an indecomposable pivotal exact
$\calc$-module, the finite tensor category $\calc_\calm^*$ of right exact $\calc$-module 
endofunctors is a pivotal tensor category.
\end{rem}

We are now in a position to introduce further structure on an exact \C-module $\calm$:
Denote by $\icoev_{c,m}\colon c \,{\to}\, \iHom(m,c\Act m)$ the unit of the adjunction
$H_m \,{\vdash}\, \iHom(m,-)$.  Let $\Sr$ be a relative Serre functor on $\calm$.
The \emph{internal trace} \Cite{Def.\,3.7}{shimi20} is the composition
  \be
  \itr_m \Colon \iHom(m,\Sr(m)) \rarr 
  {~\cong~}\iHom(m,m)^\vee \rarr{~\icoev^\vee\,} \one^\vee \,\cong\, \one \,,
  \ee
where the first isomorphism is a component of the inverse of the defining structural 
morphism of $\Sr$. The internal trace is related to a non-degenerate pairing 
$\iHom(n,\Sr(m)) \oti \iHom(m,n)\To \one$; indeed this pairing factors into
the composition of the internal Hom and the internal trace. Based on these structures
the following has been shown recently:

\begin{thm}{\rm \Cite{Thm.\,3.15}{shimi20}} \label{shimi20-Thm.3.15}
Let $(\calm,\tildepi^\calm)$ be a pivotal exact module category over a pivotal finite
tensor category $(\calc,\pi)$. Then the algebra $\iHom(m,m)$ 
in \C\ has the structure of a symmetric Frobenius algebra with Frobenius form
  \be
  \lambda_m \Colon 
  \iHom(m,m) \rarr{(\underline{\tildepi^\calm_m})_*^{}} \iHom(m,\Sr(m)) \rarr{\itr_m} \one \,,
  \ee
for any $m\iN\calm$.
\end{thm}


\section{Integrating bimodule functors to objects in the Drinfeld center}\label{sec:cifunctors}

Recall from the paragraph before Lemma \ref{lem:bimodstr} the right \C-module structures
${}^{\#\!\!}\calm$ and $\calm^\#$ that are defined on the opposite category $\calm\opp$
of the abelian category underlying a \C-module $\calm$. We have 

\begin{lem} \label{neulemma}
Let $\calm$ be a \C-module.
\Itemizeiii
\item[{\rm (i)}]
Let $G\colon {}^{\#\!\!}\calm \boti \calm \To \C$ be a bimodule functor. Then the end
$\int_{m\in\calm}G(m,m)$ has a natural structure of a comodule over the central comonad
$\,\Zc$ of \C, and can thus be seen as an object in the Drinfeld center \ZC.
\item[{\rm (ii)}]
Let $H\colon \calm^\#\boti\calm \To \C$ be a bimodule functor. Then the coend
$\int^{m\in\calm}\! H(m,m)$ has a natural structure of a module over the central monad 
$\Zm$ of \C, and can thus be seen as an object in the Drinfeld center \ZC.
\end{itemize}
\end{lem}

\begin{proof}
Abbreviate $\int_{m\in\calm}G(m,m) \,{=:}\, g$.
To obtain a candidate $\delta_g$ for the coaction of $\Zc$ on $g$ we first
concatenate the structure morphisms $j^g$ of the end and the bimodule structure of $G$ 
(which will be suppressed in the considerations below), which gives us a family
  \be
  g \rarr{j^g_{c.m}} G(c. m,c. m) \xrightarrow[\cong]{~\phi^G~} c\act G(m,m)\act c^\vee
  \label{eq:bimod-jgcm}
  \ee
of morphisms. Since $j^g$ is dinatural, this constitutes a dinatural transformation
from the object $g$ to the functor
  \be
  \begin{array}{rl}
  (\C\boti \calm)\opp \boxtimes (\C\boti\calm) & \!\! \rarr{~\,} \, \C
  \Nxl2
  \overline{(c\boti m)} \boxtimes (c'\boti m') & \!\! \longmapsto \, c'\act G(m,m')\act c^\vee .
  \eear
  \ee
Invoking the behavior of (co)ends over Deligne products
\Cite{Sect.\,3.4}{fScS2} this, in turn, gives rise to a morphism
  \be
  \bearll \dsty
  \delta_g\Colon g \rarr{} \int_{\!c \boxtimes m\in\calc\boxtimes \calm}\!
  c\act G(m,m)\act c^\vee
  \!\!& \dsty \cong \int_{\!c\in\calc}\!\int_{\!m\in\calm}\! c\act G(m,m)\act c^\vee
  \Nxl2 & \dsty
  \cong \int_{\!c\in\calc}c\act\!\int_{\!m\in\calm}\! G(m,m)\act c^\vee \,=\, \Zc(g) \,,
  \eear
  \ee
such that 
  \be
  (c\act j^g_m \act c^\vee) \circ \izc^g_c \circ \delta_g = \phi^G\circ j^g_{c.m}
  \Colon g \rarr{} c\act G(m,m)\act c^\vee
  \label{eq:jzd=j}
  \ee  
with $\izc^x_c\colon \Zc(x) \,{\to}\, c\oti x\oti c^\vee$ the structure morphism
of the central comonad.
Writing $\deltahat_c \,{:=}\, \izc_c \cir \delta_g$ for $c\iN\C$, together with
\eqref{eq:bimod-jgcm} this gives the commutative diagram
  \be
  \begin{tikzcd}[column sep=1.1em]
  G(c\Act m,c\Act m)\! \ar{rrrr}{\cong} &~ &~ &~ & \!c \act G(m,m) \act c^\vee
  \\[5pt]
  ~ & g\, \ar{ul}[xshift=5pt]{j^g_{c\Act m}} \ar{dr}[swap,xshift=3pt]{\delta_g} \ar{rr}{\deltahat_c}
  &~ & \,c \act g \act c^\vee \ar{ur}[swap,xshift=-6pt]{c\Act j^g_m \Act c^\vee} & ~ 
  \\
  ~ & ~ & \Zc(g) \ar{ur}[swap,xshift=-4pt]{\izc_c} & ~ & ~ 
  \end{tikzcd}
  \label{eq:trianglediag1}
  \ee
(We do not directly use this diagram here, but it will be instrumental in the 
proof of Lemma \ref{lem:monadmod-morphs} below.)
\\
Now denote by $\Delta\colon \Zc(c)\To \Zc^2_{}(c)$ the 
comultiplication of $\Zc$. To see the coaction property of $\delta_\bfF$ we compare two
commutative diagrams. The first of these is 
  \be
  \begin{tikzcd}
  \bfF \ar{r}{\delta_\bfF} \ar[out=300,in=179]{ddrrr}[swap,yshift=8pt]{j_{(x\otimes y).m}^\bfF}
  & \Zc(\bfF) \ar{r}{\Delta} \ar[out=310,in=179]{drr}[swap,yshift=3pt]{\izc_{x\otimes y}^\bfF}
  & \Zc^2(\bfF) \ar{r}{\izc^{\Zc(\bfF)}_x}
  & x \act \Zc(\bfF) \act x^\vee \ar{d}{x.\izc_y^\bfF.x^\vee}
  \\
  ~&~&~& (x{\otimes}y) \act \bfF \act (x{\otimes}y)^\vee
  \ar{d}{(x\otimes y).j^\bfF_m.(x\otimes y)^\vee}
  \\
  ~&~&~& (x{\otimes}y) \act G(m,m) \act (x{\otimes}y)^\vee
  \end{tikzcd}
  \label{eq:d1}
  \ee
Here the upper square commutes by the defining property 
  \be
  \begin{tikzcd}[column sep=3.1em, row sep=2.5em]
  \Zc(c) \ar{r}{\izc_{x\otimes y}^{c}} \ar{d}[swap]{\Delta} & 
  (x{\otimes}y) \act c \act (x{\otimes}y)^\vee 
  \\
  \Zc^2_{}(c) \ar{r}{\izc_x^{\Zc(c)}} & x \act \Zc(c) \act x^\vee
  \ar{u}[swap]{x. \izc^{c}_y . x^\vee}
  \end{tikzcd}
  \ee
of the comultiplication, while the lower square commutes owing to the relation 
\eqref{eq:jzd=j}. The second diagram is
  \be
  \begin{tikzcd}[column sep=1.3em,row sep=2.2em]
  \bfF \ar{r}{\delta_\bfF}
  \ar[out=272,in=179]{ddr}[swap,yshift=8pt]{j_{(x\otimes y).m}^{\bfF}}
  & \Zc(\bfF) \ar{rrr}{\Zc(\delta_\bfF)} \ar{d}{\izc^\bfF_x}
  &~&~& \Zc^2(\bfF) \ar{d}{\izc^{\Zc(\bfF)}_x}
  \\
  ~& x\act\bfF\act x^\vee \ar{rrr}{x.\delta_\bfF.x^\vee} \ar{d}{x.j^\bfF_{y.m}.x^\vee}
  &~&~& x \act \Zc(\bfF) \act x^\vee \ar{d}{x.\izc^g_y.x^\vee}
  \\
  ~& (x{\otimes}y)\act G(m,m) \act (x{\otimes}y)^\vee
  &~&~& (x{\otimes}y)\act g \act (x{\otimes}y)^\vee
  \ar{lll}[swap]{~(x{\otimes}y).j^g_m.(x\otimes y)^\vee}
  \end{tikzcd}
  \label{eq:d2}
  \ee
Here the left and the lower right square commute again due to \eqref{eq:jzd=j}, while
the upper right square commutes by the definition of $\Zc$.
Comparing the outer hexagons of the diagrams \eqref{eq:d1} and \eqref{eq:d2}
establishes the comodule property of the morphism $\delta_\bfF$ and thus proves claim (i). 
\\
Claim (ii) follows by applying claim (i) to the opposite category.  
For instance, the commutative diagram analogous to \eqref{eq:trianglediag1} reads
  \be
  \begin{tikzcd}[column sep=1.1em]
  c \act H(m,m) \,\Act \Vee c\! \ar{rrrr}{\cong} \ar{dr}[swap,xshift=3pt]{c\Act i^\bfH_m \Act\Vee c}
  &~ &~ &~ & \!H(c\Act m,c\Act m) \ar{dl}[xshift=-4pt]{i^\bfH_{c\Act m}}
  \\[5pt]
  ~ & c \act \bfH \,\Act \Vee c\, \ar{dr}[swap,xshift=3pt]{\izm_c} \ar{rr}{\rhohat_c}
  & ~ & \,\bfH &~
  \\
  ~ & ~ & \Zm(\bfH) \ar{ur}[swap,xshift=-3pt]{\rho_\bfH} & ~ & ~
  \end{tikzcd}
  \label{eq:trianglediag2}
  \ee
with $\izm^x_c\colon c\oti x\oti \Vee c \,{\to}\, \Zm(x)$ the structure morphism
of the central monad,
$\bfH \,{:=} \int^{m\in\calm}\! H(m,m)$ and $\rhohat_c \,{:=}\, \rho_\bfH \cir \izm_c$.
\end{proof}

An analogous result holds for natural transformations:

\begin{lem} \label{lem:monadmod-morphs}
Let $\calm$ be a \C-module.
\Itemizeiii
\item[{\rm (i)}]
Let $G_1,G_2\colon {}^{\#\!\!}\calm \boti \calm \To \C$ be bimodule functors
and $\nu\colon G_1\,{\rightarrow}\, G_2$ be a bimodule natural transformation. 
Then the morphism $\overline\nu \,{:=}\,\int_{m\in\calm}\nu_{m,m} \colon
\int_{m\in\calm}G_1(m,m)\to \int_{m\in\calm}G_2(m,m)$ in $\calc$ that is
induced by the functoriality of the end is even a morphism of $\,\Zc$-comodules.
\item[{\rm (ii)}]
Let $H_1,H_2\colon \calm^\#\boti\calm \To \C$ be bimodule functors and 
$\nu\colon  H_1\,{\rightarrow}\, H_2$ be a bimodule natural transformation. Then the morphism 
$\int^{m\in\calm}\!\nu_{m,m} \colon \int^{m\in\calm}\! H_1(m,m)\To \int^{m\in\calm}\! H_2(m,m)$ 
in $\calc$ induced by the functoriality of the coend is even a morphism of $\Zm$-modules.
\end{itemize}
\end{lem}

\begin{proof}
We prove claim (i); the proof of (ii) is dual.
Consider two copies of the diagram (\ref{eq:trianglediag1}), one for $G_1$ and one for
$G_2$. We can connect the top lines of these two diagrams by the 
morphisms $\nu_{c\Act m,c\Act m}\colon G_1(c\Act m,c\Act m)\To G_2(c\Act m,c\Act m)$ and
$c\act\nu_{m,m} \act c^\vee\colon  c\act G_1(m,m)\act c^\vee\To c\act G_2(m,m)\act c^\vee$. 
The resulting square commutes because $\nu$ is required to be a bimodule natural transformation.
Similarly, we can connect the second line of the diagram for $G_1$ to the second line of
the diagram for $G_2$ by the morphisms $\overline\nu\colon g_1\To g_2$ and 
$c\act \overline\nu\act c^\vee \colon c\act g_1 \act c^\vee \To c\act g_2\act c^\vee$.
The two resulting squares that involve the dinatural structure
morphisms for $g_1$ and $g_2$, respectively, commute, owing to the
definition of $\overline\nu$ and the functoriality of the \C-actions.
Thus in short, in the diagram
  \be
  \begin{tikzcd}[column sep=1.1em]
  ~ & G_2(c\Act m,c\Act m) \ar{rr}{\cong} \ar[leftarrow]{dd}[swap,yshift=-15pt]{j_{c\act m}^{g_2}}
  & ~ & c\act G_2(m,m) \act c^\vee
  \\
  G_1(c\Act m,c\Act m) \ar[crossing over]{rr}[xshift=26pt]{\cong} \ar{ur}[xshift=2pt]{\nu_{c\Act m}}
  & ~ & c\act G_1(m,m) \act c^\vee \ar{ur}[xshift=9pt]{c\Act \nu_m\Act c^\vee} & ~
  \\
  & g_2 \ar{rr}{} & ~ & c\act g_2\act c^\vee \ar{uu}[swap]{c\Act j_m^{g_2}\Act c^\vee}
  \\
  g_1\ar{ur}[xshift=1pt]{\overline\nu} \ar{uu}[xshift=1.3pt]{j_{c\act m}^{g_1}} \ar{rr}{}
  & ~& c\act g_1\act c^\vee
  \ar{ur}[swap,xshift=-7pt]{c\Act \overline\nu\Act c^\vee}
  \ar[crossing over]{uu}[yshift=-15pt]{c\Act j_m^{g_1}\Act c^\vee}
  & \\
  \end{tikzcd}
  \label{eq:diagramin tikz2}
  \ee
the left, right, front, back and top squares commute for every $m\iN\calm$. As a 
consequence, the square at the bottom of \eqref{eq:diagramin tikz2} commutes as well.
 \\
Proceeding in the same way as above, the lower triangles in the two
diagrams of type (\ref{eq:trianglediag1}) for $G_1$ and $G_2$ can be combined to
  \be
  \begin{tikzcd}[column sep=1.8em]
  ~ & g_2 \ar{rr}{\hat\delta^{g_2}_c}\ar{ddr} & ~ & \!c \act g_2\act c^\vee
  \\
  g_1\! \ar{ur}[xshift=2pt,yshift=-2pt]{\overline\nu}
  \ar[crossing over]{rr}[xshift=-5pt]{\hat\delta^{g_1}_c} \ar{ddr}
  & ~ & c \act g_1\act c^\vee \ar{ur}[xshift=4pt,yshift=-3pt]{c\act\overline\nu\act c^\vee} & ~
  \\[16pt]
  ~ & ~ & \Zc(g_2) \ar{uur}[swap,xshift=-4pt]{\izc^{g_2}_c} & ~
  \\
  ~ & \Zc(g_1) \ar{ur}[swap,xshift=-5pt,yshift=-1pt]{\Zc(\overline\nu)}
  \ar[crossing over]{uur}[xshift=4pt,yshift=-3pt]{\izc^{g_1}_c} & ~ & ~
  &&
  \end{tikzcd}
  \ee
The two triangles in this diagram commute by construction, the top square is just the
bottom square of \eqref{eq:diagramin tikz2}, and the square on the right commutes by the 
functoriality of the end in the central comonad $\Zc$. Thus the square on the
left commutes as well. This is the desired result: it states that 
$\overline \nu$ is a morphism of $\Zc$-comodules.
\end{proof}

We combine Lemma \ref{neulemma} and Lemma \ref{lem:monadmod-morphs} to

\begin{thm}\label{thm:integralprop}
Let $\calc$ be a finite tensor category and $\calm$ be a \C-module. Then the assignments
  \be
  \mbox{$\endint$}:~~ G \,\mapsto \int_{m\in\calm}\! G(m,m) \qquad\text{and}\qquad
  \mbox{$\coendint$}:~~ H \,\mapsto \int^{m\in\calm}\!\! H(m,m) 
  \ee
for $G \iN \Fun_{\C|\C} ({}^{\#\!\!}\calm \boti \calm, \C)$ and
$H \iN \Fun_{\C|\C} (\calm^\#\boti\calm, \C)$ provide functors
  \be
  \begin{array}{rl}
  \endint: & \Fun_{\C|\C} ({}^{\#\!\!}\calm \boti \calm, \C) \to \ZC 
  \Nxl4
  \text{and}\qquad \coendint: &
  \Fun_{\C|\C} (\calm^\#\boti\calm, \C) \to \ZC \,,
  \end{array}
  \label{eq:cifs}
  \ee
respectively.
\end{thm}

We call the functors \eqref{eq:cifs} the \emph{central integration functors} for 
the \C-module \M.


\section{Internal natural transformations}\label{sec:inat}

\subsection{Definition}

Let $\calc$ be a finite tensor category and $\calm$ and $\caln$ be left $\calc$-mo\-du\-les. 
It is well known that the category $\Funre_\calc(\calm,\caln)$ of right exact module 
functors is a finite category, and in fact has a natural structure of a finite $\ZC$-module 
as follows: For $(c_0,\beta_0)\iN\ZC$ and $F\iN\Funre_\calc(\calm,\caln)$ the functor
  \be
  \begin{array}{rll}
  \calm &\!\! \xrightarrow{~~} \!\!& \caln \,,
  \Nxl1
  m &\!\! \xmapsto{~~} \!\!& c_0\act F(m) \,.
  \eear
  \ee
is again right exact. Also, 
it acquires the structure of a $\calc$-module functor by the composition
  \be
  \bearll
  c_0\act F(c\act m) \rarr\cong c_0\act (c\act F(m))
  \!\!& \rarr\cong (c_0\oti c)\act F(m) 
  \Nxl2 &
  \xrightarrow[\cong]{~(\beta_0)_c^{}\act F(m)~} (c\oti c_0)\act F(m) 
  \rarr\cong c\act (c_0\act F(m)) \,,
  \label{eq:modfun}
  \eear
  \ee
where the first isomorphism is furnished by the module functor structure on $F$, 
and the second and forth use the mixed associativity constraint for $\caln$.
We write $(c_0,\beta_0)\act F\iN\Funre_\calc(\calm,\caln)$ for the module functor obtained
this way from the object $(c_0,\beta_0)\iN\ZC$ and the functor $F\iN\Funre_\calc(\calm,\caln)$.
It is straightforward to check that this prescription endows the finite functor category 
$\Funre_\calc(\calm,\caln)$ with the structure of a module over the finite tensor category \ZC.

\begin{Definition}
Let $\calc$ be a finite tensor category and let $\calm$ and $\caln$ be $\calc$-modules.
Endow the functor category $\Funre_\calc(\calm,\caln)$ with the structure of a
module over the finite tensor category $\ZC$ as described above. Given
$G,H\iN \Funre_\calc(\calm,\caln)$, we call the internal Hom
  \be
  \relNat(G,H) := \iHom_{\Funre_\calc(\calm,\caln)}(G,H) ~\iN \ZC
  \label{eq:def:relNat}
  \ee
the object of \emph{internal natural transformations} from $G$ to $H$.
\end{Definition}
 
Dually we set
  \be
  \relcNat(G,H) := \icHom_{\Funre_\calc(\calm,\caln)}(G,H) ~\iN \ZC \,.
  \label{eq:relcNat}
  \ee

\begin{rem}\label{rem:relNatascoend}
The Yoneda lemma in the form of formula \eqref{eq:delta-property}
allows one to express internal natural transformations as a coend:
  \be
  \relNat(F,G) \,= \int^{z\in\ZC}\!\! \Hom_\ZC(z, \relNat(F,G))\otik z \,\in \ZC 
  \label{eq:relNat-as-coend}
  \ee
or, equivalently,
  \be
  \relNat(F,G) \,= \int^{z\in\ZC}\!\! \Hom_{\Funre_\C(\M,\N)}(z\Act F,G)\otik z 
  \,\in \ZC
  \label{eq:relNat-as-coend-2}
  \ee
by the adjunction defining the internal Hom \eqref{eq:def:relNat}.
We denote the dinatural structure morphisms of the coend \eqref{eq:relNat-as-coend-2} by 
  \be
  \imath^{F,G}_z \colon \quad \Hom_{\Funre_\C(\M,\N)}(z\Act F,G)\otik z
  \rarr~ \int^{z'\in\ZC}\!\! \Hom_{\Funre_\C(\M,\N)}(z'\Act F,G)\otik z' .
  \label{eq:relNat-as-coend-i}	
  \ee
\end{rem}


\subsection{Description as an end: component morphisms for
            relative natural transformations}\label{sec:inat-end}

An ordinary natural transformation is a family of morphisms. As a consequence the set of 
ordinary natural transformations between any two functors
$G,H\colon \C \,{\to}\, \mathcal D$ can (for \C\ essentially small) be expressed as an end:
  \be
  \Nat(G,H) \,= \int_{\!c \in \calc} \Hom_{\mathcal D}(G(c),H(c)) \,.
  \label{eq:NatGH}
  \ee
The structure morphisms
  \be
  j^{G,H}_c \Colon \int_{\!c' \in \calc} \Hom_{\mathcal D}(G(c'),H(c')) 
  \to  \Hom_{\mathcal D}(G(c),H(c))
  \label{eq:def:j}
  \ee
of this end are just the projections to the components of the natural transformation.
As a special case, the natural transformations of the identity functor of a category \C\
give the \emph{center} $\End(\IdC) \eq \int_{\!c \in \calc} \HomC(c,c)$ of \C. The latter
provides a Morita invariant formulation of the center $Z(A)$ of an algebra $A$, according to
$\End(\Id_{A\text{-mod}}) \,{\cong}\, Z(A)$.

We are now going to show that the results of the previous subsection allow us to express
internal natural transformations as an end as well. We start by noticing that
in situations which involve module categories, it can be rewarding to replace morphism sets 
(or rather, morphism spaces) by internal Homs -- the relative Serre functors introduced in 
Definition \ref{def:4.22fScS2} provide an illustrative example.
It is thus natural to consider for any pair $G,H\colon \calm \,{\to}\,\caln$ 
of module functors  the end
  \be
  \relNatend(G,H) := \int_{\!m \in \calm} \iHom_\caln(G(m),H(m)) \,.
  \label{eq:def:relNatend}
  \ee
We denote the members of the dinatural family for this end by
  \be
  j^{G,H}_m\Colon \int_{\!m' \in \calm} \iHom_\caln(G(m'),H(m')) \xrightarrow{~~}
  \iHom_\caln(G(m),H(m)) \,.
  \label{eq:def_jm}
  \ee
Similarly to the situation for ordinary natural transformations,
the morphism $j^{G,H}_m$ in $\C$ plays the role of projecting to
the $m$th `component' $\iHom_\caln(G(m),H(m))$ of the object $\relNatend(G,H)$.

Dually, we have for $G,H\colon \calm \,{\to}\,\caln$ the coend
  \be
  \relcNatend(G,H) = \int^{m \in \calm}\!\! \icHom_\caln(G(m),H(m)) \,.
  \label{eq:def:relcoNatcoend}
  \ee
We denote its dinatural family by
  \be
  \jc^{\,G,H}_m \Colon \icHom_\caln(G(m),H(m)) \rarr{~}\!
  \int^{m' \in \calm}\!\! \icHom_\caln(G(m'),H(m')) \,.
  \ee

A crucial observation is now that, since the internal Hom and internal coHom are bimodule 
functors, by Theorem \ref{thm:integralprop} we can (and will) for $G,H\iN \Funre_\C(\M,\N)$
regard the objects \eqref{eq:def:relNatend} and \eqref{eq:def:relcoNatcoend} as objects in 
\ZC. (It should be appreciated, though, that the structure morphisms $j^{G,H}_m$ and 
$\jc^{\,G,H}_m$ are morphisms in \C\ and not in \ZC.)
We are then ready to state

\begin{thm}\label{thm:main1}
Let \C\ be a finite tensor category and \M\ and \N\ be finite \C-modules.
\\[2pt]
{\rm (i)}\, 
The end $\relNatend(G,H)\iN \ZC$ is canonically isomorphic to the
internal natural transformations: 
  \be
  \relNat(F,G) \,= \int_{m\in\M}\! \iHom(F(m),G(m)) \,.
  \ee
{\rm (ii)}\, Analogously, the internal coHom \eqref{eq:relcNat}
is canonically isomorphic to a coend:
  \be
  \relcNat(F,G) \,= \int^{m\in\M}\!\! \icHom(F(m),G(m)) \,.
  \ee
\end{thm}

\begin{proof}
We prove (i); the proof of (ii) is dual. 
	 By the adjunction that defines the internal Hom $\relNat(F,G)$, proving (i)
	 is equivalent to showing the adjunction
  \be
  \Hom_\ZC(z,\relNatend_\ZC(G,H)) \,\cong\, \Hom_{\Funre_\C(\M,\N)} (z\Act G,H) 
  \label{eq:adj2}
  \ee
for $G,H \iN \Funre_\C(\M,\N)$ and $z \iN \ZC$.
Writing $z \eq (c_0,\cb_0)$ with $c_0\iN\C$ and $\cb_0$ a half-braiding on $c_0$,
the right hand side of (\ref{eq:adj2}) can be described as follows. The data 
characterizing a module natural transformation from $z\Act G$ to $H$ are a family 
  \be
  \text{(D)} : \qquad\quad \big(\eta_{m}{\colon}\, c_0\Act G(m) \To H(m) 
  \big)_{m\in\calm} 
  \ee
of morphisms in $\caln$, indexed by elements in $m\iN\calm$, and
subject to two types of conditions:
 \\[-1.55em]
\begin{itemize}
\item[(C1)] Naturality:
For every morphism $m\Rarr{f} m'$ in $\calm$ the diagram 
  \be
  \begin{tikzcd}
  c_0\act G(m) \ar{r}{\eta_{m}^{}} \ar{d}[swap]{c_0\act G(f)} & H(m) \ar{d}{H(f)}
  \\
  c_0\act G(m') \ar{r}[]{\eta_{m'}}& H(m')
  \end{tikzcd}
  \ee
in $\caln$ commutes.
		
\item[(C2)] Module natural transformation:
With $\phi^H$ the datum turning $H$ into a module functor, 
for every $c\iN\calc$ and $m\iN\calm$ the diagram 
  \be
  \begin{tikzcd}[column sep=3.2em]
  c_0\act G(c\act m) \ar{d}[swap]{\Phi_{c_0\Act G}} \ar{r}{\eta^{}_{c\Act m}}
  & H(c\act m) \ar{d}{\phi^H}
  \\
  (c{\otimes}c_0)\act G(m) \ar{r}{c\Act\eta^{}_{m}} & c\act H(m)
  \end{tikzcd}
  \ee
in $\caln$ commutes, where the morphism $\Phi_{c_0\Act G}$ is the module functor datum 
$\phi^G$ for $G$, followed by a half-braiding.
\end{itemize}
An element of the morphism space on the left hand side of \eqref{eq:adj2} is a morphism 
  \be
  \widetilde\eta \Colon (c_0,\cb_0) \to \relNatend_\ZC(G,H)
  \ee
to an end in $\calc$. After post-composing with the structure morphisms $j^{G,H}_{m}$
of that end, this amounts to a dinatural family of morphisms in $\calc$, labeled by
objects $m\iN\calm$. The data of this family are
  \be
  \text{(D}') : \qquad\quad \big(\widetilde\eta_{m} \,{:=}\,
  j^{G,H}_{m} \cir \widetilde\eta{\colon}\, c_0\To
  \iHom(G(m),H(m)\big)_{m\in\calm} \,,
  \ee
and they are subject to two constraints:
 \\[-1.5em]
\begin{itemize}
\item[(C1$'$)] Dinaturality:
For every morphism $m\Rarr{f} m'$ in $\calm$ the diagram 
  \be
  \begin{tikzcd}[column sep=2.8em, row sep=3.1em]
  c_0 \ar{rr}{\widetilde\eta_{m}} \ar{d}[swap]{\widetilde\eta_{m'}} & ~ &
  \iHom(G(m),H(m)) \ar{d}{\iHom(G(m),H(f))}
  \\
  \iHom(G(m'),H(m')) \ar{rr}{\iHom(G(f),H(m'))} && \iHom(G(m),H(m')) 
  \end{tikzcd}
  \ee
in $\calc$ commutes.
		
\item[(C2$'$)] 
Compatibility with the half-braiding:
For every $c\iN\calc$, the diagram 
  \be
  \begin{tikzcd}[column sep=2.8em]
  c_0\oti c \ar{d}[swap]{\beta_0^{-1}}\ar{r}{\tilde\eta_{c_0}\otimes c} & \relNatend(G,H)
  \oti c \ar{d}
  \\
  c\oti c_0 \ar{r}{c\otimes\tilde\eta_{c_0}}& c\oti \relNatend(G,H)
  \end{tikzcd}
  \label{eq:C2'}
  \ee
commutes, where the right downwards arrow is the component at $c$ of the
distinguished half-braiding on $\relNatend(G,H)\iN\ZC$.
\end{itemize}
To compare the two sides of \eqref{eq:adj2}, first notice that the adjunction defining
the internal Hom for the \C-module \N\ gives natural isomorphisms
  \be
  \Hom_{\caln}(c_0\act G(m),H(m)) \rarr\cong \Hom_\calc(c_0,\iHom(G(m),H(m))
  \label{eq:theadj}
  \ee
for all $m\iN\calm$. This adjunction maps data of type (D) to data of type (D$'$). We are
going to show that also the respective conditions on these data are mapped to each other.

\smallskip
\noindent
Thus consider
condition (C1), i.e.\ the equality $H(f)\cir\eta_{m} \,{=}\, \eta_{m'}\cir c_0\act G(f)$ 
of morphisms in $\Hom_{\caln}(c_0\act G(m),H(m'))$ for all morphisms $m\Rarr{f} m'$.
By definition of $\iHom(G(m),H(f))$, the internal Hom adjunction maps
$G(f)\cir\eta_{m}$ to $\iHom(G(m),H(f)) \cir \widetilde\eta_{m}$ with
$\widetilde\eta_{m}$ is the image of $\eta_{m}$ under the adjunction \eqref{eq:theadj}.
A similar argument applies to pre-composition, showing that $\eta_{m'}\cir c_0\act G(f)$ is
mapped by \eqref{eq:theadj} to $\iHom(G(f),H(m') \cir \widetilde\eta_{m}$.
Together it follows that indeed condition (C1) is mapped to condition (C1$'$), so that
$(\text{C1}) \,{\Leftrightarrow}\, (\text{C1}')$.

\smallskip
\noindent
Next we pick an object $m\iN\calm$ and post-compose the two composite morphisms in the
commuting diagram (C2$'$) with the canonical morphism $c\oti j^{G,H}_m$ of the end, 
thereby obtaining morphisms in $\Hom_\C(c_0\oti c, c \oti \iHom(G(m),H(m))$. Now take 
the upper-right composite morphism in \eqref{eq:C2'}. We can use the right dual of 
$c$ to consider equivalently a morphism
  \be
  c_0 \xrightarrow{~\widetilde \eta~} \relNatend(G,H) \xrightarrow{~~~}
  c\otimes \iHom(G(m),H(m)) \otimes c^\vee .
  \label{eq:morph-cv}
  \ee
Here the second morphism can be recognized as the one we used to get the structure of
a comodule over the central comonad on $\relNatend(G,H)$. Hence 
the morphism \eqref{eq:morph-cv} can be written as
  \be
  \bearll
  c_0 \xrightarrow{~\widetilde \eta~} \relNatend(G,H)
  \!\! & \xrightarrow{~j^{G,H}_{c\Act m}~} \iHom(G(c\act m),H(c\act m))
  \Nxl2 &
  \xrightarrow{~\iHom((\phi^G)^{-1}_{},\phi^H)~} c\otimes \iHom(G(m),H(m)) \otimes c^\vee
  \eear
  \ee
(here we suppress the bimodule functor structure of $\iHom$).
By the definition of $\widetilde\eta_m$, this morphism is nothing but
  \be
  c_0 \xrightarrow{~\widetilde\eta_{c\Act m}~} \iHom(G(c\act m),H(c\act m))
  \xrightarrow{~\iHom((\phi^G)^{-1}_{},\phi^H)~} c\otimes \iHom(G(m),H(m)) \otimes c^\vee .
  \ee
Under the internal-Hom adjunction this morphism is mapped to
  \be
  c_0\act(c\act G(m)) \xrightarrow{~c_0\Act(\phi^G)^{-1}_{}~} c_0\act G(c\act m)
  \xrightarrow{~\eta_{c\Act m}~} H(c\act m) \xrightarrow{~\Phi^H~} c\act H(m) \,.
  \label{mor1}
  \ee
Applying similar arguments to
the lower-left composite morphism in \eqref{eq:C2'} gives the morphism
  \be
  c_0 \xrightarrow{~~~} c\otimes c_0\otimes c^\vee 
  \xrightarrow{~c\otimes \widetilde\eta\otimes c^\vee~} c \otimes \relNatend(G,H) \otimes c^\vee
  \xrightarrow{~~~} c\otimes \iHom(G(m),H(m)) \otimes c^\vee ,
  \ee
where the first morphism is obtained by combining the half-braiding $\beta_0$ with 
the coevaluation of $c$. This is nothing but
  \be
  c_0 \xrightarrow{~~~} c\otimes c_0\otimes c^\vee
  \xrightarrow{~c\otimes \widetilde\eta_m\otimes c^\vee~}
  c\otimes \iHom(G(m),H(m)) \otimes c^\vee ,
  \ee
and is thus under the internal-Hom adjunction mapped to the morphism
  \be
  c_0\act(c\act G(m)) \,\cong\, (c_0\oti c)\act G(m)
  \xrightarrow{~\beta_0\Act G(m)~} (c\oti c_0)\act G(m)
  \xrightarrow{~c\Act\eta_m~} c\act H(m) \,.
  \label{mor2}
  \ee
Recalling the definition of $\Phi_{c_0\Act G}$ in terms of $\phi^G$ and the 
half-braiding $\beta_0$, we see that equality of the morphisms (\ref{mor1}) 
and (\ref{mor2}) is precisely the commuting diagram (C2). Thus we have
established also the equivalence $\text{(C2}) \,{\Leftrightarrow}\, \text{(C2}')$. 
\end{proof}

It is instructive to express the situation considered in the proof of Theorem
\ref{thm:main1} schematically: We have a commuting diagram
  \be
  \begin{tikzcd}[column sep=1.1em]
  \Hom_\ZC((c_0,\cb_0),\relNatend_\ZC(F,G)) \ar{rr}{\eqref{eq:adj2}}[swap]{\cong} \ar{d}
  & ~ & \Hom_{\Funre_\C(\M,\N)} ((c_0,\cb_0)\Act F,G) \ar{d}[swap]{} 
  \\
  \prod_{m} \Hom_\calc(c_0,\iHom(F(m),G(m))) \ar{rr}{ \cong} & ~ &
  \prod_{m} \Hom_{\caln}(c_0.F(m), G(m))
  \end{tikzcd}
  \label{oursquare}
  \ee
Here the left downwards arrow is post-composition by the structure morphisms of the end
\eqref{eq:def:relNatend} i.e.\ maps $f$ to $j^{F,G}_{m}\cir f$ for some $m \iN \calm$. The right downwards arrow 
comes from the fact that a natural transformation is a family of morphisms. The lower 
horizontal arrow is component-wise the internal Hom adjunction.

\begin{Example}
For $\caln \eq \calm$ and $G\eq H \eq \IdM$, the object of internal natural transformations is
  \be
  \FM := \relNat(\IdM,\IdM) = \int_{\!m \in \calm} \iHom_\calm(m,m) 
  \, \in\ZC \,.
  \ee
In particular, for \C\ as a module category over itself, this is 
  \be
  \FC =  \int_{\!c\in\calc} \iHom_\calc(c,c) = \int_{\!c\in\calc} c \oti c^\vee
  \, \in\ZC \, .
  \label{cardybulk} \ee
If \C\ is semisimple, this is the object $\bigoplus_{i} x_i^{} \oti x_i^{\vee}$, with
the summation being over the finitely many isomorphism classes of simple objects, and
with half-braiding as given e.g.\ in \Cite{Thm.\,2.3}{balKi}. It is natural to refer 
to the object $\FM$ in $\ZC$ as the \emph{center} of the \C-module $\calm$.
\end{Example}

\begin{rem}\label{rem:LR-RF}
Given a functor $F\iN\Funre_\calc(\calm,\caln)$, define a functor
  \be
  \begin{array}{rll}
  L^F:\quad \ZC &\!\!\xrightarrow{~~~}\!\!& \Funre_\calc(\calm,\caln) \,,
  \Nxl2
  z & \!\!\xmapsto{~~~}\!\! & z \act F
  \end{array}
  \ee
In the special case that $\calm \eq \caln \eq {}_\calc\calc$ is $\calc$ seen
as a module over itself, we have a canonical identification
$ \Funre_\calc(\calc,\calc) \,{\cong}\,\, \calc$, under which
$L^{\Id_\calc}\colon \ZC\To \C$ is the forgetful functor. It follows from the proof 
of Theorem \ref{thm:main1} that the right adjoint  of $L^F$ is the functor
  \be
  \begin{array}{rll}
  R^F :\quad \Funre_\calc(\calm,\caln) &\!\!\xrightarrow{~~~}\!\!& \ZC \,,
  \Nxl2
  G &\!\!\xmapsto{~~~}\!\!& \int_{m\in\calm}\iHom_{\caln}(F(m),G(m)) = \relNat(F,G) \,.
  \eear
  \label{eq:def:RF}
  \ee
In the special case $\calm \eq \caln \eq {}_\calc\calc$ as well as $F \eq \Id_\calc$ 
the functor \eqref{eq:def:RF} is given by
  \be
  c_0 \,\longmapsto \int_{\!c\in\calc} \iHom(c,c\oti c_0) 
  \,\cong \int_{\!c\in\calc} c\otimes c_0 \oti c^\vee ,
  \ee
where we use Remark \ref{rem:rem4} and the central comonad (\ref{ccomonad}). 
In other words, the adjunction \eqref{eq:adj2} can be regarded as a generalization of
the adjunction that defines the central comonad.
\end{rem}

An interesting consequence of Theorem \ref{thm:main1} is obtained when combining it
with the fact \cite{schau13} that, for $\calc$ a finite tensor category and $\calm$
a $\calc$-module, there is an explicit braided equivalence
  \be
  \theta_\calm \Colon \ZC \to \ZCM \,.
  \ee
Let us compute the object $\theta_{\calm}(\relNat(F,G)) \iN \ZCM)$
for $F,G\iN\Funre_\calc(\calm,\calm)$. To this end we use the fact that
a right exact functor admits a right adjoint, and that the
right adjoint of the forgetful functor $U_{\calm}\colon \ZCM\To\calcm$, which is exact,
is the coinduction functor associated with the central comonad on $\calcm$.

\begin{prop}
Let $\calc$ be a finite tensor category and $\calm$ a $\calc$-module.
For $F,G\iN\Funre_\calc(\calm,\calm)$ we have 
  \be
  \theta_{\calm}(\relNat(F,G)) \cong \widetilde I(G\cir F\ra) \, \in\ZCM \,,
  \ee
where $F\ra$ is the right adjoint of $F$ and 
  \be
  \begin{array}{rcl}
  \widetilde I \Colon \calcm &\!\!\xrightarrow{~~~}\!\!& \ZCM 
  \Nxl2 
  \varphi &\!\!\xmapsto{~~~}\!\!&
  \int_{\psi\in\calcm}\psi\ra\cir\varphi\cir\psi
  \eear
  \ee
is the right adjoint of the forgetful functor $U_{\calm}$.
\end{prop}

\begin{proof}
Applying the composition of $\theta_\calm$ with the forgetful functor
$U_\calm\colon \ZCM\To\calcm$ to the object $(c_0,\beta_0)\iN\ZC$ of
the Drinfeld center gives the \C-module endofunctor $(c_0,\beta_0)\act \Id_\calm$, 
which as a functor is given by acting with $c_0$ and has a module functor structure 
given by $\beta_0$. Pre-composing with $F\iN\Funre_\calc(\calm,\calm)$ yields
  \be
  (U_\calm\cir \theta_\calm)(c_0,\beta_0) \circ F = L^F(c_0,\beta_0) \,,
  \ee
with $L^F$ as introduced in Remark \ref{rem:LR-RF}. By the adjunction in Remark \ref{rem:LR-RF}
we thus have
  \be
  \bearll
  \Hom_{\calcm}((U_\calm\cir \theta_\calm)(c_0,\beta_0) \cir F,G) \!\!&
  \cong \Hom_\ZC((c_0,\beta_0),R^F(G)) 
  \Nxl2 &
  = \Hom_\ZC((c_0,\beta_0),\relNat(F,G))
  \Nxl2 &
  \cong \Hom_{\ZCM}(\theta_\calm (c_0,\beta_0),\theta_\calm(\relNat(F,G))) \,,
  \eear
  \ee
where the last isomorphism holds because $\theta_\calm$ is an equivalence.
It follows that for all $\varphi\iN\ZCM$ we have
  \be
  \begin{array}{rl}
  \Hom_{\ZCM}(\varphi,\theta_{\calm}(\relNat(F,G))) \!\!\!&
  \cong \Hom_{\calcm}(U_{\calm}(\varphi) \circ F,G)
  \Nxl2 &
  \cong \Hom_{\calcm}(U_{\calm}(\varphi),G\cir F\ra_{})
  \Nxl2 &
  \cong \Hom_{\ZCM}(\varphi,\widetilde I(G\cir F\ra_{})) \,.
  \eear
  \ee
Here the first isomorphism uses the definition \eqref{eq:def:RF} of $R^F$ together with
the adjunction we just derived (and again the fact that $\theta_\calm$ is an equivalence).
\end{proof}		

\begin{rem}
Specifically for the case that $F \eq G \eq \Id_\calm$, we find that 
  \be
  \theta_\calm (\relNat(\Id_\calm,\Id_\calm)) \,\cong \int_{\!\psi\in\calcm}\!\!\psi\ra\cir\psi \,.
  \ee
Comparing this formula with \eqref{cardybulk}, we can rephrase this by saying that
\emph{after application of Schauenburg's equivalence $\theta_\M$, the center of any module
category is diagonal.}
	(In the application to two-dimensional conformal field theory alluded to at the
	end of the Introduction, this implies that the bulk state space of any full
	conformal field theory becomes diagonal when regarded not as an object
	of $\ZC$, but as an object in the equivalent category $\ZCM$.)
\end{rem}


\subsection{Compositions}\label{sec:inat-compositions}

Ordinary natural transformations can be composed horizontally as well as vertically.
Both compositions are conveniently described component-wise. 
A \emph{vertical} composition of \emph{internal} natural transformations
clearly exists, being just a particular instance the multiplication of internal Homs.
In this subsection we introduce in addition a \emph{horizontal} composition
of internal natural transformations. We also describe their vertical composition from a
different perspective. As we will see, these compositions are again naturally formulated 
in terms of components. Indeed, the constructions can largely be performed in analogy
with those for ordinary natural transformations, including an Eck\-mann-Hil\-ton argument.

We start by observing that for the \ZC-module $\Funre_\calc(\calm,\caln)$ the natural evaluation
$\iev$ \eqref{eq:def:iev} of internal Homs, which is used to obtain their multiplication,
is a natural transformation 
  \be
  \iev_{F,G}:\quad  \relNat(F,G)\act F \to G 
  \label{eq:ievFG}
  \ee
between module functors in $\Funre_\C(\M,\N)$. Under the defining adjunction isomorphism 
of the internal Hom, $\iev_{F,G}$ is induced from the identity morphism on $\relNat(F,G)$.
Owing to $\relNat \,{\cong}\, \relNatend$ the latter is a morphism whose codomain is an
end, so that it can be described as a dinatural family
$\big( \relNat(F,G) \To \iHom(F(m),G(m)) \big)_{m\in\calm}$,
and this family is just the structure morphism of the end.
Thus the components of the evaluation $\iev_{F,G}$ are just the images in
  \be
  \big( \relNat(F,G)\Act F(m) \To G(m) \big)_{m\in\calm}
  \label{eq:iev-family}
  \ee
of the structure morphisms of the end under the internal Hom adjunction for the module
category \M\ over \C. The associative multiplication of internal Homs 
$\relNat(-,-)$ is a family of morphisms
  \be
  \imuver \equiv \imuver(G_1,G_2,G_3)
  \Colon \relNat(G_2,G_3)\otimes \relNat(G_1,G_2) \rarr{~} \relNat(G_1,G_3) 
  \label{eq:def:imuver}
  \ee
in $\ZC$ obeying the standard associativity condition, and as described in
\eqref{eq:unit4muiHom} we have units 
  \be
  \relid_F \in \Hom_\ZC(\one_\ZC,\relNat(F,F)) \,.
  \ee


\begin{Definition}
Let $\calm$ and $\caln$ be $\calc$-modules and $G_1,G_2,G_3\iN \Funre_\calc(\calm,\caln)$.
The morphism $\imuver$ from $\relNat(G_2,G_3)\oti \relNat(G_1,G_2)$ to $\relNat(G_1,G_3)$
that is introduced in \eqref{eq:def:imuver}
is called the \emph{vertical composition} of internal natural transformations.
\end{Definition}

The following result justifies this terminology:

\begin{prop} \label{prop:compos-inherited}
Let $\calm$ and $\caln$ be $\calc$-modules and $G_1,G_2,G_3\iN \Funre_\calc(\calm,\caln)$.
Consider for $m\iN\calm$ the composition
  \be
  \bearl
  \alpha_m\Colon \relNat(G_2,G_3) \otimes  \relNat(G_1,G_2) 
  \Nxl2 \hspace*{7em}
  \xrightarrow{\,j^{G_2,G_3}_m\otimes j^{G_1,G_2}_m}\,
  \iHom_{\caln}(G_2(m),G_3(m))\otimes \iHom_{\caln}(G_1(m),G_2(m) )
  \Nxl2 \hspace*{17.1em}
  \xrightarrow{\,\muiHom^{}_{G_1(m),G_2(m),G_3(m)}\,}\, \iHom_{\caln}(G_1(m),G_3(m)) \,.
  \eear
  \label{eq:muver3}
  \ee
We have 
  \be
  \alpha_m = j^{G_1,G_3}_m \circ \imuver
  \ee
for any $m\iN\calm$. 
 \\
Moreover, for any $F \iN \Funre_\C(\calm,\caln)$ the unit 
$\relid_F \iN \Hom_\ZC(\one_\ZC,\relNat(F,F))$ satisfies
  \be
  j^{F,F}_{m}\circ \relid_F
  = \relid_{F(m)} \,\in \Hom_\C\big(\one, \iHom_{\caln}(F(m),F(m))\big) \,.
  \ee
for all $m\iN\calm_1$
\end{prop}

Put differently, when thinking about the structure morphisms $j^{G,G'}_m$ \eqref{eq:def_jm}
of the end \eqref{eq:def:relNatend} as projections to components, the composition $\imuver$ of
internal natural transformations is nothing but
the ordinary vertical composition of natural transformations.

\begin{proof}
Under the adjunction,
the image in $\Hom_{\Funre_\C(\M,\N)} (\relNat(G_2,G_3) \oti \relNat(G_1,G_2) \Act G_1, G_3)$ 
of $\imuver(G_1,G_2,G_3) \iN \Hom_\ZC(\relNat(G_2,G_3) \oti \relNat(G_1,G_2),\relNat(G_1,G_3)$
is given, by the definition of composition of internal Homs, by the composite
  \be
  \relNat(G_2,G_3) \otimes  \relNat(G_1,G_2)\act G_1 \xrightarrow{~\id\otimes \iev~}
  \relNat(G_2,G_3) \act G_2 \xrightarrow{~\iev~} G_3
  \ee
(we suppress the mixed associator of the module category $\caln$). The components of this
module natural transformation are of the form appearing in the lower right corner of the
diagram \eqref{oursquare}. On the other hand, the morphism $\alpha_m$ is of the type that
appears in the lower left corner of that diagram. We must show that they are related by
applying the internal-Hom adjunction to each component in the direct product. But this is
indeed the case because, as noted above (see \eqref{eq:iev-family}), the evaluation is related
to the structure maps of the end by the internal-Hom adjunction. Hence the claim follows.
 \\
The assertion about the unit $\relid_F$ follows directly from the definitions: By the
adjunction \eqref{eq:adj2} it is related to $\id_F\iN \Hom_{\Funre_\C(\M,\N)}(\one\act F,F)$.
In terms of the diagram \eqref{oursquare}, we have the identity morphism in the component
$\Hom_{\caln}(\one\act F(m), F(m))$ of the natural isomorphism for every $m \iN \M$.
\end{proof} 

Next we note that a \ko-linear category $\mathcal D$ can be seen as a module category 
over the monoidal category $\vect_\ko$. For a $\vect_\ko$-module the Hom and internal
Hom coincide, so that the internal-Hom adjunction 
gives, for each pair $d,d'\iN\mathcal D$ of objects, a natural evaluation
  \be
  \evk_{d,d'}:\quad \Hom_{\mathcal D}(d,d') \otimes_\ko d \rarr~ d' 
  \ee
as the image of the identity map under the linear isomorphism
  \be
  \Hom_{\mathcal D}( \Hom_{\mathcal D}(d,d')\otik d, d')
  \,\rarr\cong\, \Hom_{\mathcal D}(d,d')^*_{} \otimes_\ko \Hom_{\mathcal D}(d,d') \,.
  \ee
It should be appreciated that the composition of $\Hom_{\mathcal D}$ as an internal Hom for 
$\mathcal D$ as a $\vect_\ko$-module and the ordinary composition of morphisms in 
$\mathcal D$ coincide.

This observation allows us to give the following convenient description of the evaluation
$\iev_{F,G}\colon \relNat(F,G)\act F\to G$ in \eqref{eq:ievFG}:
	      Invoking from Remark \ref{rem:relNatascoend}
	      the expression \eqref{eq:relNat-as-coend-2} for $\relNat(F,G)$, we get
  \be
  \begin{array}{rl}\dsty
  \iev_{F,G} = \int^{z\in\ZC}\!\! \evk_{z\Act  F,G}: &\dsty
  \Big( \int^{z\in\ZC}\!\! \Hom_{\Funre_\C(\M,\N)}(z\Act F,G)\otik z \Big) \act F 
  \Nxl1 &\dsty
  \quad\qquad = \int^{z\in\ZC}\!\! \Hom_{\Funre_\C(\M,\N)}(z\Act F,G) \otik z\Act F 
  \,\rarr{~}\, G \,.
  \eear
  \ee
Here in the equality we use that the action functor is exact, and by 
$\int^{z\in\ZC}\! \evk_{z\Act F,G}$ we denote the morphism out of the coend 
$\int^{z\in\ZC}\! \Hom_{\Funre_\C(\M,\N)}(z\Act F,G) \otik z\Act F$ that is defined by the 
family $\big( \evk_{z\Act  F,G}\colon \Hom_{\Funre_\C(\M,\N)}(z\Act F,G)\otik z\Act F \To G
\big)^{}_{\!z\in\ZC}$ which is dinatural in $z\iN\ZC$.

\medskip

Next we define modified vertical and horizontal compositions of relative natural 
To set the stage for the horizontal composition, we formulate

\begin{Lemma} \label{lem:zGzF=zzGF}
Let \M\ and \N\ be module categories over a finite tensor category \C, and let
$F,G \iN \Funre_\calc(\calm,\caln)$. Then for any pair $z,z'$ of objects in \ZC\ 
the module functor constraint of $G$ induces an isomorphism
  \be
  (z \act G)\circ (z' \act F) \rarr\cong (z \,{\otimes_\ZC}\, z') \act (G\cir F) 
  \ee
of \C-module functors.
\end{Lemma}

\begin{proof}
For any $m \iN \M$ there is an isomorphism 
  \be
  (z \act G) \cir (z' \act F)(m) = z \act G(z'\Act F(m))
  \rarr\cong (z \,{\otimes_\ZC}\, z') \act (G\cir F)(m) 
  \ee
of functors. Due to the fact that the braiding is natural and thus compatible with 
the module functor datum $\phi^G$, this is even an isomorphism of \C-module functors.
\end{proof}

This result allows us to give

\begin{defi}
\Itemizeiii
\item[{\rm (i)}]
Let $\calm$ and $\caln$ be finite module categories over a finite tensor category $\calc$.
For any triple $F,G,H\iN\Funre_\calc(\calm,\caln)$ the \emph{modified vertical composition}
  \be
  \begin{array}{r}
  \tmuver:\quad \Hom_{\Funre_\C(\M,\N)}(z\Act G,H) \otimes_\ko
  \Hom_{\Funre_\C(\M,\N)}(z'\Act F,G) \hspace*{7em}
  \Nxl3
  \rarr~ \Hom_{\Funre_\C(\M,\N)}((z{\otimes}z')\act F,H) 
  \eear
  \label{eq:def:tmuver}
  \ee
of natural transformations is defined by
  \be
  \tmuver(\beta,\alpha):\quad
  (z\oti z')\act F\rarr\cong z\act(z'\act F) \rarr{z\Act\alpha} z\act G \rarr\beta H 
  \ee
for $z,z' \iN \ZC$.
\item[{\rm (ii)}]
Let $\calm_1,\calm_2$ and $\calm_3$ be finite module categories over a finite tensor
category $\calc$. For $z,z' \iN \ZC$, $F,F'\iN\Funre_\calc(\calm_1,\calm_2)$ and
$G,G'\iN \Funre_\calc(\calm_2,\calm_3)$ the \emph{modified horizontal composition}
  \be
  \begin{array}{r}
  \tmuhor:\quad \Hom_{\Funre_\C(\M_2,\M_3)}(z\Act G,G') \otimes_\ko
  \Hom_{\Funre_\C(\M_1,\M_2)}(z'\Act F,F') \hspace*{6em}
  \Nxl3
  \rarr~ \Hom_{\Funre_\C(\M_1,\M_3)}((z{\otimes}z')\act G\cir F,G'\cir F') 
  \eear
  \ee
of natural transformations is defined to be the composition of the ordinary 
horizonal composition with the isomorphism given in Lemma \ref{lem:zGzF=zzGF}.
\end{itemize}
\end{defi}

\begin{rem} \label{rem:tmuver-muver}
Admittedly, the modified vertical composition $\tmuver$ looks somewhat unnatural.
But it should be appreciated that by using the duality we can identify
  \be
  \Hom_{\Funre_\C(\M,\N)}(z\act G,H) \cong \Hom_{\Funre_\C(\M,\N)}(G,z^\vee\Act\, H) \,.
  \ee
Upon this identification, $\tmuver$ just becomes the ordinary vertical composition.
\end{rem}

\begin{rem}\label{rem:einsen}
In the special case $G \eq F$ and $z' \eq \one_{\ZC}$, the identity natural
transformation in $\Hom_{\Funre_\C(\M_1,\M_2)}(\one_\ZC\act F,F)$ is a unit 
$\widetilde\eta^\text{ver}_F$ for the modified vertical composition $\tmuver$.
 \\
Similarly, for $\M_3\eq\M_2$, $G\eq G'\eq \Id_{\M_2}$ and $z \eq \one_\ZC$, 
the identity natural transformation in
$\Hom_{\Funre_\C(\M_2,\M_2)}(\one_\ZC\act \Id_{\M_2},\Id_{\M_2})$ is a unit 
$\widetilde\eta^\text{hor}_{\M_2}$ for the modified horizontal composition $\tmuhor$. It
should be appreciated that whenever both units are defined, they are the same, 
$\widetilde\eta^\text{hor}_\M \eq \widetilde\eta^\text{ver}_{\id_\M}$.
\end{rem}

\begin{prop}
Let $\calm$ and $\caln$ be finite module categories over a finite tensor category $\calc$, 
and let $F,G,H\iN\Funre_\calc(\calm,\caln)$. The vertical composition
$\imuver\colon \relNat(G,H)\oti\relNat(F,G)\To \relNat(F,H)$ defined in \eqref{eq:def:imuver}
is the morphism out of the coend 
  \be
  \hspace*{-1.1em}
  \begin{array}{l}
  \relNat(G,H)\oti \relNat(F,G) 
  \Nxl2 \dsty
  \qquad = \int^{z\in\ZC}\!\! \Hom_{\Funre_\C(\M,\N)}(z\Act G,H) \otik z
  \,\otimes_\ZC^{} \!\int^{z'\in\ZC}\!\! \Hom_{\Funre_\C(\M,\N)}(z'\Act F,G) \otik z'
  \Nxl1 \dsty
  \qquad \cong \int^{z_,z'\in\ZC}\!\! \Hom_{\Funre_\C(\M,\N)}(z\Act G,H) \otik 
  \Hom_{\Funre_\C(\M,\N)}(z'\Act F,G) \otik (z{\otimes_\ZC}z')
  \eear
  \label{eq:natGHnatFG}
  \ee
that is given by the family 
  \be
  \begin{array}{l}
  \Hom_{\Funre_\C(\M,\N)}(z\Act G,H) \otik 
  \Hom_{\Funre_\C(\M,\N)}(z'\Act F,G) \otik (z{\otimes_\ZC}z')
  \Nxl2 \dsty
  \hspace*{9em} \rarr{\,\tmuver\,}
  \Hom_{\Funre_\C(\M,\N)}(z{\otimes_\ZC}z')\act F,H) \otimes_\ko (z{\otimes_\ZC}z')
  \Nxl2 \dsty
  \hspace*{9em} \rarr{\imath^{F,H}_{z\otimes z'}}
  \int^{\zeta\in\ZC}\! \Hom_{\Funre_\C(\M,\N)}(\zeta\act F,H)\otik \zeta \,=\,\relNat(F,H)
  \eear
  \label{eq:...=natFH}
  \ee
of morphisms in \ZC\
$($with $\imath^{F,H}$ the dinatural family defined in \eqref{eq:relNat-as-coend-i}$)$
which is dinatural in $z,z'\iN\ZC$.
\end{prop}

\begin{proof}
We write $\hmuver$ for the morphism out of the coend \eqref{eq:natGHnatFG} that is defined 
by \eqref{eq:...=natFH}.
We have to compare the morphisms
  \be
  \big( \relNat(G,H)\oti\relNat(F,G)\big) \act F
  \rarr{\hmuver\Act \id_F} \relNat(F,H) \act F \rarr{\iev} H 
  \label{eq:mor1}
  \ee
and
  \be
  \relNat(G,H)\act \big(\relNat(F,G) \act F \big)
  \rarr{\id\act\iev}\, \relNat(G,H) \act G \rarr{\iev} H \,.
  \label{eq:mor2}
  \ee
Now the morphism \eqref{eq:mor1} can be expressed in terms of the family
  \be
  \begin{array}{l}
  \Hom_{\Funre_\C(\M,\N)}(z\Act G,H) \otik 
  \Hom_{\Funre_\C(\M,\N)}(z'\Act F,G) \otik (z{\otimes_\ZC}z') \act F
  \Nxl2 \dsty
  \hspace*{8.5em} \rarr{\,\tmuver\Act \id_F\,}
  \Hom_{\Funre_\C(\M,\N)}((z{\otimes_\ZC}z')\act F,H) \otimes_\ko (z{\otimes_\ZC}z' \act F)
  \Nxl1 \dsty
  \hspace*{8.5em} \rarr{\iev_{(z\otimes z')\Act F,H}^{}} H
  \eear
  \label{eq:GHxFG.F-H-1}
  \ee
that is dinatural in $z,z'\iN\ZC$, while the morphism \eqref{eq:mor2} is described by the family
  \be
  \begin{array}{l}
  \Hom_{\Funre_\C(\M,\N)}(z\Act G,H) \otik z \,\otimes_\ZC
  \Hom_{\Funre_\C(\M,\N)}(z'\Act F,G) \otik z' \act F
  \Nxl2 \dsty
  \hspace*{8.4em} \rarr{\id\oti\iev_{z'\!\Act F,G}}\,
  \Hom_{\Funre_\C(\M,\N)}(z\Act G,H) \otik z \act G \,\rarr{\iev_{z\Act G,H}^{}}\, H \,.
  \eear
  \label{eq:GHxFG.F-H-2}
  \ee
Since upon use of dualities, $\tmuver$ boils down to the ordinary vertical composition
(see Remark \ref{rem:tmuver-muver}),
the two composites \eqref{eq:GHxFG.F-H-1} and \eqref{eq:GHxFG.F-H-2} coincide. It thus
follows that $\hmuver$ indeed describes the vertical composition $\imuver$
of internal natural transformations.
\end{proof}

\begin{rem}
We describe again the unit of the vertical composition. It is the morphism
in $\Hom_\ZC(\one_\ZC,\relNat(F,F))$ that is given by the composite
  \be
  \bearll
  \one_\ZC \!& \rarr{~~}\, \Hom_{\Funre_\C(\M,\N)}(\one_\ZC\act F,F) \otik \one_\ZC
  \Nxl2 &\dsty
  \rarr{\imath^{F,F}_{\one_\ZC}}\int^{z\in\ZC}\!\!\Hom_{\Funre_\C(\M,\N)}(z\act F,F)\otik z \,,
  \eear
  \ee
where the first morphism is determined by the unit $\widetilde\eta^\text{ver}_F\iN
\Hom_{\Funre_\C(\M,\N)}(\one\act F,F)$, while the second is the structure morphism
of the coend.
\end{rem}

We are now in a position to introduce also a \emph{horizontal} composition
of internal natural transformations:

\begin{defi}
Let $\calm_1,\calm_2$ and $\calm_3$ be finite module categories over a finite tensor category
$\calc$, and let $F,F'\iN\Funre_\calc(\calm_1,\calm_2)$ and
$G,G'\iN \Funre_\calc(\calm_2,\calm_3)$. The \emph{horizontal composition}
  \be
  \relNat_{\Funre_\calc(\calm_2,\calm_3)}(G,G') \otimes \relNat_{\Funre_\calc(\calm_1,\calm_2)}
  (F,F') \to \relNat_{\Funre_\calc(\calm_1,\calm_3)} (G\cir F,G'\cir F')
  \ee
is the morphism $\imuhor$ out of the coend
  \be
  \hspace*{-1.2em}
  \begin{array}{l}
  \relNat(G,G')\otimes \relNat(F,F')
  \Nxl2 \dsty
  \quad = \int^{z\in\ZC}\!\! \Hom_{\Funre_\C(\M_2,\M_3)}(z\Act G,G') \otik z
  \,\otimes_\ZC \!\int^{z'\in\ZC}\!\! \Hom_{\Funre_\C(\M_1,\M_2)}(z'\Act F,F') \otik z'
  \Nxl1 \dsty
  \quad \cong \int^{z_,z'\in\ZC}\!\! \Hom_{\Funre_\C(\M_2,\M_3)}(z\Act G,G') \otik
  \Hom_{\Funre_\C(\M_1,\M_2)}(z'\Act F,F') \otik (z{\otimes_\ZC}z')
  \eear
  \ee
that is given by the family 
  \be
  \begin{array}{l}
  \Hom_{\Funre_\C(\M_2,\M_3)}(z\Act G,G') \otik
  \Hom_{\Funre_\C(\M_1,\M_2)}(z'\Act F,F') \otik (z{\otimes_\ZC}z')
  \Nxl2 \dsty
  \hspace*{5.7em} \rarr{~\tmuhor~}\,
  \Hom_{\Funre_\C(\M_1,\M_3)}\big( (z{\otimes_\ZC}z')\act (G\cir F),G'\cir F' \big)
  \otimes_\ko (z{\otimes_\ZC}z') \quad
  \Nxl1 \dsty
  \hspace*{5.7em} \xrightarrow{\,\imath^{G\circ F,G'\circ F'}_{z\otimes z'}}
  \int^{\zeta\in\ZC}\!\! \Hom_{\Funre_\C(\M_1,\M_3)}(\zeta\act (G\cir F),G'\cir F')\otik \zeta
  \Nxl1 \multicolumn1r {\dsty
  =\,\relNat(G\cir F,G'\cir F') }
  \eear
  \label{eq:outofcoend2}
  \ee
of morphisms in \ZC, which is dinatural in $z,z'\iN\ZC$.
\end{defi}

The so defined horizontal composition $\imuhor$ and the vertical composition $\imuver$
satisfy the Eckmann-Hil\-ton property:

\begin{prop}\label{prop:EH4imuhor+imuver}
Let $\calm_1,\calm_2$ and $\calm_3$ be finite module categories over a finite tensor 
category $\calc$. Then for any two triples $F,G,H\iN\Funre_\calc(\calm_1,\calm_2)$ and
$F',G',H'\iN \Funre_\calc(\calm_2,\calm_3)$, the diagram
  \be
  \hspace*{-1.5em}
  \begin{tikzcd}[column sep=-4.4 em, row sep=2.7 em]
  \relNat(G',H') \oti \relNat(G,H) & ~& ~
  \\[-3.1em]
  \hspace*{6.4 em} \oti \relNat(F',G')\oti \relNat(F,G)
  \ar{rr} \ar{dd}[swap]{\imuhor\otimes\imuhor}
  & ~ & \relNat(G',H')\oti\relNat(F',G')
  \\[-3.1em]
  ~ & ~ & \hspace*{6.4 em} \oti \relNat(G,H) \oti \relNat (F,G)
  \ar{d}{\imuver\otimes\imuver}
  \\
  ~ \relNat(G'\cir G,H'\cir H)\otimes \relNat (F'\cir F,G'\cir G) 
  \ar{dr}[swap]{\imuver}
  & ~ & \relNat(F',H')\otimes \relNat(F,H) \ar{dl}{\imuhor}
  \\
  ~& \relNat(F'\cir F,H'\cir H) & ~
  \end{tikzcd}
  \ee
commutes. Here the horizontal arrow is the braiding in $\ZC$. 
\end{prop}

\begin{proof}
When describing all horizontal and vertical compositions in the diagram in terms of
dinatural families analogous to \eqref{eq:...=natFH} and \eqref{eq:outofcoend2},
the statement follows directly from the standard properties of vertical and horizontal
compositions of module natural transformations.
\end{proof}

As usual, the Eckmann-Hilton property allows one to derive commutativity.

\begin{cor}
Let $\calm$ be a finite module category over a finite tensor category $\calc$.
Then the algebra $\relNat(\Id_\calm,\Id_\calm)$ in \ZC\ is braided commutative.
\end{cor}

\begin{proof}
Put all six functors in Proposition \ref{prop:EH4imuhor+imuver} to be the identity functor 
of \M.  Then comparison with Remark \ref{rem:einsen} tells us that the 
horizontal and vertical composition have the same unit.
\end{proof}


\section{Pivotal module categories and Frobenius algebras}\label{sec:e+p}

So far we have been dealing with finite tensor categories and their (bi)module 
categories. In terms of modular functors, these structures are naturally related
to a \emph{framed} modular functor \cite{doSs3,fScS4}. 
To have a relation with an \emph{oriented} modular functor, 
additional algebraic structure is required, in particular the finite tensor categories 
should come with a pivotal structure. As a consequence, the algebras arising as
internal Ends will have the additional structure of a Frobenius algebra.
In this section we study relative natural transformations in such a setting.

Thus we now suppose that \C\ is a \emph{pivotal} finite tensor category and that \M\ is an
\emph{exact} \C-module. Without loss of generality we then further assume that \C\ is
strict pivotal, so that the Nakayama functor of \M\ is an ordinary module functor
and ${}^{\#\!\!}\M \eq \M^\# \,{=:}\, \overline\M$, see Remark \ref{rem:ngsc}. In this 
situation, the central integration functors $\endint$ and $\coendint$ appearing in Theorem 
\ref{thm:integralprop} are both functors from $\Fun_{\C|\C} (\overline\calm \boti \calm, \C)$ 
to \ZC. We now show that they are actually related by the module Eilenberg-Watts equivalences
\eqref{eq:moduleEW}:

\begin{prop}\label{prop:intpivotal}
Let $\calc$ be a pivotal finite tensor category and $\calm$ be an exact $\calc$-module.
Then the two functors $\endint,\coendint\colon
\Fun_{\C|\C} (\overline\calm \boti \calm, \C)  \To \ZC$ satisfy
  \be
  \mbox{$\coendint$}
  = \mbox{$\endint$} \,\circ \big( \Id_{\overline\calm}\boti \Nr_\M\big)
  \ee
with $\Nr_\M$ the right exact Nakayama functor of \M.
\end{prop}

\begin{proof}
Since $\calm$ is exact, all module and bimodule functors with domain 
$\overline\calm \boti \calm$ are exact. For
$G\iN \Fun_{\C|\C} (\overline\calm \boti \calm, \C) $ we therefore have
  \be
  \int_{\!m\in\calm}\!\! G(\overline m\boti m)
  \cong G\big( \mbox{$\int_{m\in\calm}^{}$}\, \overline m\boti m \big) \quad~\text{and}\quad~
  \int^{m\in\calm}\!\!\! G(\overline m\boti m)
  \cong G\big( \mbox{$\int^{m\in\calm}_{}$}\, \overline m\boti m \big)
  \label{eq:twoisos}
  \ee
as isomorphisms of objects in $\calc$.
Now for any object $\mu \iN \calz(\overline \calm\boti\calm)$ with underlying object 
$\dot\mu\iN \overline\M\boti\M$, the object $G(\dot\mu)$ comes with a canonical 
half-braiding, given by $ G(\dot\mu) \oti c \,{\cong}\, G(\dot \mu\act c) 
$\linebreak[0]${\cong}\, G(c\act {\dot\mu}) \,{\cong}\, c \oti G(\dot \mu)$ for $c\iN\C$.
Moreover, the objects $\int_{m\in\calm} \overline m\boti m$ and 
$\int^{m\in\calm} \overline m\boti m$ of \C\ are endowed with natural balancings
\Cite{Cor.\,4.3}{fScS2}, i.e.\ they are in fact objects of $\calz(\overline \calm\boti\calm)$.
It follows that the isomorphisms in \eqref{eq:twoisos} are even isomorphisms in \ZC.
Analogous isomorphisms in \ZC\ are valid when $\overline m\boti m$ is replaced by 
$\overline m\boti H(m)$ for any module endofuctor $H$.
 \\[2pt]
Hence it remains to show that there is an isomorphism
  \be
  \int^{m\in \calm}\!\! \overline m\boti m
  \,\cong \int_{\!m\in \calm}\! \overline m\boti \Nr_\calm(m) 
  \label{eq:redclaim}
  \ee
of objects in $\calz(\overline \calm\boti\calm)$. That this is indeed the case follows from the two-sided adjoint equivalences \eqref{eq:moduleEW}.
	Indeed, when considering the identity
functor on $\calm$ as a left exact functor with trivial module structure, one has
$ \Phile(\Id_\M) \eq \int^{m\in\calm} \overline m\boti m \iN \calz(\overline\calm\boti\calm)$;
the isomorphism \eqref{eq:redclaim} is then obtained by recalling the definition 
	\eqref{eq:def:Nr} of the Nakayama functor together with the fact that
      the Eilenberg-Watts functors $\Phire$ and $\Psire$ are quasi-inverses.
\end{proof}

Of particular interest to us is a statement that follows from Proposition \ref{prop:intpivotal}
together with the following result (for which \C\ is not required to be pivotal):

\begin{lem}\label{lem:ausgelagert}
Let $\calc$ be a finite tensor category and $\calm$ and $\caln$ be exact \C-module
categories. Let $F,G\colon\M\to \N $ be module functors. Then there is an isomorphism
  \be
  \int^{m\in\M}\!\! \icHom(F(m),G(m)) \,\cong \int^{m\in\M}\! \iHom(\Sl_{\N}{\circ}F(m),G(m))
  \label{eq:coends=}
  \ee
as objects in $\ZC$. In particular, the coends on both sides of this equality exist.
\end{lem}

\begin{proof}
For any pair of module functors $F,G\colon\M\to \N $, the two functors
$\icHom(F,G)$ and $\iHom(\Sl_{\N}\,{\circ}\,F,G)$
are bimodule functors from $\calm^\#\boti\calm$ to \C. That they are actually 
bimodule functors follows from the properties of the internal Hom and coHom together with 
the twisted bimodule property of the left relative Serre functor \Cite{Lemma\,4.23}{fScS2}
which is analogous to \eqref{eq:Sr(cm)=cvv.Sr(m)}.
By combining the defining property \eqref{eq:leftSerre} of a left relative Serre functor and
the relation \eqref{eq:icHom-iHom} between internal Hom and coHom, we obtain,
for any pair $F,G$ of module functors, an isomorphism
  \be
  \icHom(F,G) \cong \iHom(\Sl_{\N}\,{\circ}\,F,G)
  \label{icHom-iHom-Sl}
  \ee
of bimodule functors. By Theorem \ref{thm:integralprop}, this isomorphism of functors
implies an isomorphism of their coends \eqref{eq:coends=} as objects in $\ZC$. 
\end{proof}

\begin{thm}\label{thm:exactovercenter}
Let $\calc$ be a pivotal finite tensor category and $\calm$ and $\caln$ be exact \C-modules.
Then the functor category $\Funre_\C(\calm,\caln)$ is an exact module category over \ZC.
\end{thm}

\begin{proof}
Applying Proposition \ref{prop:intpivotal}
to the bimodule functor $\iHom(\Sl_{\N}\,{\circ}\,F,G)$ we obtain an isomorphism
  \be
  \int^{m\in\M}\!\! \iHom(\Sl_{\N}{\circ}F(m),G(m))
  \,\cong \int_{\!m\in\M}\! \iHom(\Sl_{\N}{\circ}F(m),G\,{\circ}\,\Nr_{\M}(m))
  \ee
of objects in $\ZC$.  
The left and right Nakayama functors of a finite linear category form an adjoint pair 
\Cite{Lemma\,3.16}{fScS2}. Using that
$\Sl_{\N} \eq D^{-1}\Act\, \Nl_{\N} \eq D^\vee\Act\, \Nl_{\N}$
with $D$ the distinguished invertible object of \C\ 
(see Remark \ref{rem:S-vs-N}), it follows that there are isomorphisms
  \be
  \iHom(\Sl_{\N}{\circ}F,G\,{\circ}\,\Nr_{\M}) 
  \,\cong\, \iHom(\Nl_{\N} \,{\circ}\, F, D\act G\,{\circ}\,\Nr_{\M})
  \,\cong\, \iHom(F,\Nr_{\N}\,{\circ}\,D\act G\,{\circ}\,\Nr_{\M})
  \ee
of functors. Combining this statement with Lemma \ref{lem:ausgelagert} we find an isomorphism
  \be
  \int^{m\in\M}\!\! \icHom(F(m),G(m)) \,\cong 
  \int_{\!m\in\M}\! \iHom(F(m),\Nr_{\N}\,{\circ}\,D\act G\,{\circ}\,\Nr_{\M}(m)) 
  \label{eq:iH-icH-Nr}
  \ee
as objects in $\ZC$. This means that the functor $\Nr_\N \cir (D \act -) \cir \Nr_\M$
is a relative Serre functor for $\Funre_\C(\calm,\caln)$. This shows in particular that
$\Funre_\C(\calm,\caln)$ is an exact module category over $\ZC$.
\end{proof}

We will use the notation
  \be
  \NrR := \Nr_\N \circ (D \act -) \circ \Nr_\M
  \label{eq:def:NrR}
  \ee
for the relative Serre functor for $\Funre_\C(\calm,\caln)$ obtained in the proof.

\begin{rem}
If $\C$ is unimodular, then $\NrR \eq \Nr_\N \cir (-) \cir \Nr_\M$ is the 
right Nakayama functor $\Nr_{\Funre(\M,\N)}$ of the category $\Funre(\M,\N)$ of
all functors from \M\ to \N\ \Cite{Lemma\,3.21}{fScS2}.
\end{rem}
	
\medskip

It is instructive to check explicitly that the functor \eqref{eq:def:NrR} satisfies
the following two properties which befit the relative Serre functor of 
$\Funre_\C(\calm,\caln)$ (taking, for perspicuity, \C\ just to be
pivotal, rather than strict pivotal):
\\[2pt]
(1) $\NrR$ is a twisted \ZC-module functor, i.e.\ there are coherent natural isomorphisms
  \be
  \phi^\Funre_{z,F}:\quad \NrR(z.F) \,\xrightarrow{~\cong~} z^{\vee\vee} \Act\, \NrR(F)
  \label{eq:tmp5}
  \ee
for $z\iN\ZC$ and $F\iN\Funre_\C(\calm,\caln)$.
\\[2pt]
(2) $\NrR$ is an endofunctor of $\Funre_\C(\calm,\caln)$. That is, for $F\iN\Funre_\C(\calm,\caln)$,
the functor $\NrR(F)$ is again in $\Funre_\C(\calm,\caln)$, i.e.\ is a (non-twisted) module 
functor: there are coherent natural isomorphisms
  \be
  \NrR(F)(c\act m) \xrightarrow{~\cong~} c \act \big( \NrR(F)(m) \big)
  \ee
for $c\iN\calc$ and $m\iN\calm$. 
	
Let us first check that the functor $\NrR$ sends \C-module functors to \C-module 
functors. For $F\iN\Funre_\C(\calm,\caln)$, $c\iN\calc$ and $m\iN\calm$ we have
  \be
  \begin{array}{rl}
  \NrR(F)(c\act m) \,\equiv \! & \Nr_{\N} \circ (D \act F) \circ \Nr_\M(c\act m)
  \Nxl2 & \xrightarrow{~\cong~}
  \Nr_{\N} \circ (D \act F) \big( {}^{\vee\vee\!}c \act \Nr_\M(m) \big)
  \Nxl2 & \xrightarrow{~\cong~}
  \Nr_{\N} \circ \big( D \act ({}^{\vee\vee\!}c \act F \cir \Nr_\M(m) \big)
  \Nxl2 & \xrightarrow{~\cong~}
  \Nr_{\N} \big( (D \oti {}^{\vee\vee\!}c) \act F \cir \Nr_\M(m) \big)
  \Nxl2 & \xrightarrow{~\cong~}
  \Nr_{\N} \big( (c^{\vee\vee} \oti D) \act F \cir \Nr_\M(m) \big)
  \Nxl2 & \xrightarrow{~\cong~}
  {}^{\vee\vee\!}c^{\vee\vee} \act \big( \Nr_{\N} \cir (D \act F) \cir \Nr_\M(m) \big)
  \Nxl2 & ~=~ c \act \big( \Nr_{\N} \cir (D \act F) \cir \Nr_\M(m) \big)
  ~\equiv~ c \act \big( \NrR(F)(m) \big) \,.
  \eear
  \ee
Here we first use that $\Nr_\M$ is a twisted module functor, then that $F$ is a module
functor, then the module constraint, then the fact that by the Radford $S^4$-theorem 
the quadruple right dual of \C\ is naturally isomorphic to conjugation by $D$, and
finally that $\Nr_\N$ is a twisted module functor.
The same type of calculation, only one step shorter, shows that $\NrR$ is a properly twisted
\ZC-module functor: For $F\iN\Funre_\C(\calm,\caln)$ and $z\in\ZC$ we have
(writing $z \eq (\dot z,\beta)$)
  \be
  \begin{array}{rl}
  \NrR(z\act F) \,\equiv \! & \Nr_{\N} \circ (D \act (z\act F)) \circ \Nr_\M
  \xrightarrow{~\cong~}
  \Nr_{\N} \circ \big((D \oti \dot z)\act F \big) \circ \Nr_\M 
  \Nxl2 & \xrightarrow{~\cong~}
  \Nr_{\N} \circ (\dot z^{\vee\vee\vee\vee} \oti D)\act F) \circ \Nr_\M
  \Nxl2 & \xrightarrow{~\cong~}
  \dot z^{\vee\vee} \act \big( \Nr_{\N} \cir (D\act F) \cir \Nr_\M \big)
  ~\equiv~ z^{\vee\vee} \act \big( \NrR(F) \big) \,.
  \eear
  \ee

So far, we have imposed the requirement of being pivotal on the finite tensor category 
\C. Now assume further that the \C-modules \M\ and \N\ are
pivotal as well, with respective pivotal structures $\pi^\M$ and $\pi^\N$. 
The pivotal structure $\pi^\M$ gives a family of isomorphisms
$\pi^\M_m \colon m \,{\rarr{}}\, \Sr_\M(m) \eq D\act \Nr_\M(m)$ which are twisted
module natural transformations, and analogously for $\pi^\N$. In particular, for any 
functor $F\iN \Funre_\C(\M,\N)$ and any object $m\iN\M$ we can form the composite
  \be
  \begin{array}{rl}
  F(m) \!\! & \rarr{F(\pi^\M_m)} F(D\act \Nr_\M(m))
  \Nxl2 &
  \rarr{\pi^\N_{F(D\Act \Nr_\M(m)) }} D\act\Nr_\N \cir F \cir (D\act \Nr_\M)(m)
  \,\cong\, D\oti {}^{\vee\vee\!}D \Act\, \Nr_\N \cir F \cir \Nr_\M(m) \,.
  \end{array}
  \ee
In case \C\ is \emph{unimodular}, i.e.\ $D \eq \one$, this gives us a family of isomorphisms
  \be
  \pi^\N_{F\circ\Nr_\M(m)} \circ F(\pi^\M_m) : \quad 
  F(m) \rarr\cong \Nr_\N \cir F \cir \Nr_\M(m) = \NrR(F)(m) 
  \ee	  
which provides an isomorphism
  \be
  \Id_{\Funre_\C(\calm,\caln)} \rarr\cong \NrR
  \label{eq:FtoNrR(F)}
  \ee
of endofunctors of $\Funre_\C(\calm,\caln)$.
We then arrive at  

\begin{cor}\label{cor:pivotalovercenter}
Let \C\ be a unimodular pivotal finite tensor category and $\calm$ and $\caln$ be pivotal
module categories over \C. Then the functor category $\Funre_\C(\calm,\caln)$ 
has a structure of a pivotal module category over the pivotal finite tensor category $\ZC$.
\end{cor}


\begin{proof}
By unimodularity of \C\ we have $\Nr_\M \eq \Sr_\M$ and $\Nr_\N \eq \Sr_\N$.
As in the proof of Proposition \ref{prop:intpivotal} let us take, without loss of generality,
the pivotal structure of \C\ to be strict. Then upon using the isomorphisms $\Sr_\M \To \Id_\M$
and $\Sr_\N \To \Id_\N$ that come from the pivotal structure on $\M$ and $\N$, respectively,
to trivialize 
the relative Serre and thus Nakayama functors of \M\ and \N, the consistency condition
that according to Definition \ref{def:Mpivotal} has to be met in order for the isomorphism 
\eqref{eq:FtoNrR(F)} to be a pivotal structure on $\Funre_\C(\calm,\caln)$ reduces to a
combination of identities and is thus trivially satisfied. We refrain from spelling out the
explicit form that the consistency condition takes after inserting the unique isomorphisms
involved in those trivializations, as it does not add any further insight.
\end{proof}

By combining Corollary \ref{cor:pivotalovercenter} with Theorem \ref{shimi20-Thm.3.15} we 
further get

\begin{cor}
Let \C\ be a unimodular pivotal finite tensor category, let \M\ and \N\ be exact \C-modules
with pivotal structures, and let $F\colon \calm\To\caln$ a module functor. Then
the algebra $\relNat(F,F)$ has a natural structure of a symmetric Frobenius algebra in the Drinfeld center \ZC. In particular, $\relNat(\Id_\M,\Id_\M)$ 
has a natural structure of a commutative symmetric Frobenius algebra.
\end{cor}


\vskip 3.5em

\noindent
{\sc Acknowledgements:}\\[.3em]
JF is supported by VR under project no.\ 2017-03836. CS is partially supported by the
RTG 1670 ``Mathematics inspired by String theory and Quantum Field Theory''
and by the Deutsche Forschungsgemeinschaft (DFG, German Research Foundation)
under Germany's Excellence Strategy - EXC 2121 ``Quantum Universe'' - 390833306.

  \newpage


\newcommand\wb{\,\linebreak[0]} \def\wB {$\,$\wb}
\newcommand\Bi[2]    {\bibitem[#2]{#1}}
\newcommand\inBo[8]  {{\em #8}, in:\ {\em #1}, {#2}\ ({#3}, {#4} {#5}), p.\ {#6--#7} }
\newcommand\inBO[9]  {{\em #9}, in:\ {\em #1}, {#2}\ ({#3}, {#4} {#5}), p.\ {#6--#7} {\tt [#8]}}
\newcommand\J[7]     {{\em #7}, {#1} {#2} ({#3}) {#4--#5} {{\tt [#6]}}}
\newcommand\JO[6]    {{\em #6}, {#1} {#2} ({#3}) {#4--#5} }
\newcommand\JP[7]    {{\em #7}, {#1} ({#3}) {{\tt [#6]}}}
\newcommand\Jpress[7]{{\em #7}, {#1} {} (in press) {} {{\tt [#6]}}}
\newcommand\BOOK[4]  {{\em #1\/} ({#2}, {#3} {#4})}
\newcommand\PhD[2]   {{\em #2}, Ph.D.\ thesis #1}
\newcommand\Prep[2]  {{\em #2}, preprint {\tt #1}}
\newcommand\uPrep[2] {{\em #2}, unpublished preprint {\tt [#1]}}
\def\adma  {Adv.\wb Math.}
\def\ajse  {Arabian Journal for Science and Engineering}
\def\alnt  {Algebra\wB\&\wB Number\wB Theory}          
\def\alrt  {Algebr.\wb Represent.\wB Theory}         
\def\apcs  {Applied\wB Cate\-go\-rical\wB Struc\-tures}
\def\comp  {Com\-mun.\wb Math.\wb Phys.}
\def\coma  {Con\-temp.\wb Math.}
\def\cpma  {Com\-pos.\wb Math.}
\def\duke  {Duke\wB Math.\wb J.}
\def\imrn  {Int.\wb Math.\wb Res.\wb Notices}
\def\jims  {J.\wb Indian\wb Math.\wb Soc.}
\def\jams  {J.\wb Amer.\wb Math.\wb Soc.}
\def\joal  {J.\wB Al\-ge\-bra}
\def\joms  {J.\wb Math.\wb Sci.}
\def\jopa  {J.\wb Phys.\ A}
\def\jktr  {J.\wB Knot\wB Theory\wB and\wB its\wB Ramif.}
\def\jpaa  {J.\wB Pure\wB Appl.\wb Alg.}
\def\mama  {ma\-nu\-scripta\wB mathematica\wb}
\def\mams  {Memoirs\wB Amer.\wb Math.\wb Soc.}
\def\momj  {Mos\-cow\wB Math.\wb J.}
\def\nupb  {Nucl.\wb Phys.\ B}
\def\pams  {Proc.\wb Amer.\wb Math.\wb Soc.}
\def\pspm  {Proc.\wb Symp.\wB Pure\wB Math.}
\def\quto  {Quantum Topology}
\def\sigm  {SIGMA}
\def\taac  {Theo\-ry\wB and\wB Appl.\wb Cat.}
\def\tams  {Trans.\wb Amer.\wb Math.\wb Soc.}
\def\toap  {Topology\wB Applic.}
\def\topo  {Topology}
\def\trgr  {Trans\-form.\wB Groups}

\end{document}